\documentclass[a4paper]{article}

\usepackage[T1]{fontenc}
\usepackage[latin1]{inputenc}
\usepackage[english,francais]{babel}
\usepackage{graphicx}
\usepackage[all]{xy}
\usepackage{amsmath}
\usepackage{amsthm}
\usepackage{mathrsfs,amsfonts,amssymb}
\usepackage{enumerate}
\usepackage{fullpage}
\usepackage{hyperref}
\hypersetup{linktoc=page,%
   pdfstartview=FitH,%
   pdfauthor={B.~Michel},%
   pdftitle={Masse des opérateurs GJMS},%
   pdfkeywords={Géométrie conforme, Conformal geometry, Opérateur GJMS, GJMS operators,%
      Métrique de Habermann-Jost, Habermann-Jost metric,%
      fonction de Green, Green function, masse, mass},%
}

\newcommand{\cad}{c'est-\`a-dire}
\newcommand{\ssi}{si et seulement si}
\newcommand{\resp}{respectivement}
\newcommand{\da}{développement asymptotique}
\newcommand{\dl}{développement limité}

\newcommand{\haut}[1]{ {}^{#1}}
\newcommand{\bas}[1]{ {}_{#1}}
\newcommand{\abs}[1]{\left\lvert #1\right\rvert}
\newcommand{\norm}[1]{\left\lVert #1\right\rVert}
\newcommand{\set}[1]{\left\{ #1\right\}}
\newcommand{\wh}{\widehat}
\newcommand{\wt}{\widetilde}

\newcommand*{\lcoset}[2]{\left.\raisebox{-0.3ex}{$#1$}%
   \backslash\raisebox{0.3ex}{$#2$}\right.}

\newcommand{\vphi}{\varphi}

\newcommand{\ud}{\mathrm{d}}

\newcommand{\N}{\mathbb{N}}
\newcommand{\Z}{\mathbb{Z}}

\newcommand{\R}{\mathbb{R}}
\newcommand{\C}{\mathbb{C}}

\DeclareMathOperator{\Id}{Id}

\DeclareMathOperator{\tr}{tr}

\newcommand{\Sph}{\mathbb{S}}

\newcommand{\sym}{\mathrm{sym}}

\DeclareMathOperator{\Pan}{P}
\newcommand{\Rho}{\mathsf{P}}
\newcommand{\Iota}{\mathsf{J}}
\newcommand{\G}{\Gamma}
\newcommand{\Scal}{\mathrm{Scal}}
\newcommand{\Rm}{\mathrm{R}}
\newcommand{\Ric}{\mathrm{Ric}}
\DeclareMathOperator{\Vol}{\mathrm{Vol}}
\newcommand{\vol}{v}

\newcommand{\eucl}{\mathrm{eucl}}

\newcommand{\delr}{\partial_{r}}

\newcommand{\Homg}{\mathcal{H}}
\newcommand{\grandO}{\mathcal{O}}
\newcommand{\clap}{\mathfrak{c}}

\newcommand{\be}{\begin{equation}}
\newcommand{\ee}{\end{equation}}
\newcommand{\bea}{\begin{eqnarray*}}
\newcommand{\eea}{\end{eqnarray*}}

\newcommand{\bi}{\begin{itemize}}
\newcommand{\ei}{\end{itemize}}
\newcommand{\itbul}{\item[$\bullet$]}

\numberwithin{equation}{section}

\theoremstyle{plain}
\newtheorem{theo}{Théorème}[section]
\newtheorem{prop}[theo]{Proposition}
\newtheorem{coro}[theo]{Corollaire}
\newtheorem{lemm}[theo]{Lemme}

\theoremstyle{definition}
\newtheorem{defi}[theo]{Défintion}
\newtheorem*{notation}{Notations}

\theoremstyle{remark}
\newtheorem{rema}[theo]{Remarque}
\newtheorem*{merci}{Remerciements}

\theoremstyle{plain}
\newtheorem*{theoIdvpG}{Théorème \ref{theodvpG}}
\newtheorem*{theoIcovconf}{Théorème \ref{theocovconf}}
\newtheorem*{theoImassepos}{Théorème \ref{theomassepos}}
\newtheorem*{theoIinvas}{Théorème \ref{theoinvas}}

\author{B.~Michel\thanks{I3M -- Université Montpellier 2 -- France.
 E-mail: \texttt{benoit.michel@math.univ-montp2.fr}}}
\title{Masse des opérateurs GJMS}
\date{}

\begin{document}

\maketitle

\selectlanguage{francais}

\begin{abstract} On généralise la construction  de Habermann et Jost de métrique canonique dans une classe
conforme Yamabe-positive, faite à l'aide de la singularité de la fonction de Green du laplacien conforme.

En dimension $n=2k+1$, $2k+2$ ou $2k+3$, si le $k$-ième opérateur GJMS $P_k$ admet une fonction
de Green, on montre qu'en restriction à de bons choix de facteurs conformes, le terme constant de la singularité
de celle-ci est une densité conforme de poids
$2k-n$. S'il est positif, on l'utilise pour définir une métrique canonique dans la classe conforme. Dans le cas de
l'opérateur de Paneitz-Branson $P_2$, en dimension $5$, $6$ ou $7$, on
montre un résultat de positivité. Dans le cas général, on le relie également à un invariant asymptotique de la variété
obtenue par projection stéréographique \emph{via} la fonction de Green.
\end{abstract}

\selectlanguage{english}
\begin{abstract} This article generalizes a construction by Habermann and Jost of a canonical metric in a
Yamabe-positive conformal class, which uses the Green function of the conformal Laplacian.
In dimension $n=2k+1$, $2k+2$, or $2k+3$, if the $k$-th GJMS operator $\Pan_k$ admits a Green function,
the constant term of its singularity is shown to be a conformal density of weight $2k-n$, when restricted to
appropriate choices of conformal factor. When it is positive, it
is used to build a canonical metric in the conformal class.  In the case of the Paneitz-Branson operator
$\Pan_2$, in dimension $5$, $6$ or $7$, we show a positiveness result. In additition, we relate it to an asymptotic
invariant of the manifold obtained by stereographic projection \emph{via} the Green function.
\end{abstract}
\selectlanguage{francais}

\section{Introduction} \label{sec.massepk.intro}

Si l'approche de Fefferman et Graham \cite{FeffermanGraham-confinv,FeffermanGraham-ambmet}
s'avère très fructueuse pour les problèmes locaux en géométrie
conforme, la définition de métriques \og{}canoniques\fg{} dans une classe conforme donnée peut apporter beaucoup
au point de vue global. Par \og{}canonique\fg{}, on entend \og{}définie en fonction de la classe conforme\fg{} de sorte qu'une
équivalence conforme entre variétés soit une isométrie des métriques canoniques associées. De telles métriques
permettent par exemple de définir \emph{à la} Weil-Petersson une métrique sur l'espace des modules des classes
conformes, voir \cite{HabermannJost-methabjost} et les références qui y sont citées.

Dans le cas où l'invariant de Yamabe
\[
Y([g_{0}]) =\inf_{g\in [g_{0}]}\frac{\left(\int_{M}\Scal^{g} \ud\vol_{g}\right)}{\Vol(M,g)^{\frac{n-2}{n}}}
\]
(où $\Scal^{g}$ désigne la courbure scalaire de $g$) est négatif ou nul, il existe une unique métrique de volume
$1$ réalisant le minimum, ce qui donne un exemple de métrique canonique.

Dans le cas $Y([g_0])>0$, il n'y a plus unicité en général (voir \cite{Schoen-varscal}). Habermann et Jost
définissent alors dans \cite{HabermannJost-methabjost} une autre métrique canonique, en dimension $3$
et dans le cas localement conformément localement plat
(généralisant d'autres constructions, voir les références dans \cite{HabermannJost-methabjost, Habermann-LN}).
Ils utilisent pour cela le terme constant du \da{} au voisinage de la diagonale de la fonction de Green du laplacien
conforme $\Pan_1^g =\Delta_g +\frac{n-2}{4(n-1)}\Scal^g$.
Habermann \cite{Habermann-LN} étend la définition aux dimensions $4$ et $5$, en utilisant
un contrôle de la singularité de la fonction de Green de $\Pan_1$ permis par l'introduction de certaines coordonnées
normales conformes par Lee et Parker \cite{LeeParker-Yamabeprob}.

L'invariance conforme de la métrique de Habermann-Jost est principalement
due à la covariance conforme de $\Pan_1$, qui implique une covariance similaire pour sa fonction de Green. Il
est donc naturel de tenter d'adapter la construction à la famille d'opérateurs, dits GJMS, introduite par Graham,
Jenne, Mason et Sparling dans \cite{GJMS}. C'est l'objet principal de cet article.\\

On se place sur une variété différentielle compacte $M$, munie d'une classe con\-forme de métriques riemanniennes
$[g_0]$. Pour une métrique $g$ de $[g_0]$, le $k$-ième opérateur GJMS $\Pan_k^g$ est un
opérateur différentiel, défini pout tout $k$
entier si la dimension $n$ est impaire, et pour $k\leq n/2$ en dimension paire.
Sa propriété fondamentale est la loi de transformation sous les changements
conformes de métriques: si $\tilde{g}=\vphi^{\frac{4}{n-2k}}g$ est une métrique conforme%
\footnote{Cette forme du facteur conforme exclut la cas $n=2k$, qui n'est pas étudié ici. 
Voir les remarques \hyperref[rqdvpG1]{\ref*{rqdvpG}}.} à $g$,
$$\Pan_k^{\tilde{g}}u=\vphi^{-\frac{n+2k}{n-2k}}\Pan_k^g (\vphi u)\,.$$
L'opérateur $\Pan^g_k$ est par ailleurs une déformation de la puissance $k$-ième du laplacien par des
termes de plus bas degré, ce qui en fait un opérateur elliptique. Il admet donc des paramétrix, et il est naturel
de faire l'hypothèse qu'il est inversible en tant qu'opérateur $C^\infty(M)\to C^\infty(M)$.
Dans ce cas, son inverse admet un noyau intégral $G_k^g$, que l'on appelle \emph{fonction de Green} de $\Pan_k^g$.
C'est une fonction $C^\infty$ sur $M\times M$ privée de sa diagonale. On s'intéresse à la singularité
de $G_k^g(p,.)$ en un point $p\in M$.

La covariance conforme de $\Pan_k^g$ implique que si $\tilde{g}=\vphi^{\frac{4}{n-2k}}g$, la fonction de Green se
transforme selon:
$$G_k^{\tilde{g}}(x,y) =\frac{G_k^g(x,y)}{\vphi(x)\vphi(y)}\,.$$
Cette équivariance permet de faciliter le calcul de la singularité
de $G_k^g$ en choisissant convenablement $g$ dans sa classe conforme. D'après Lee et Parker
\cite{LeeParker-Yamabeprob}, on peut par un changement conforme fixer arbitrairement les
valeurs du tenseur de Ricci et de ses dérivées covariantes symétrisées en un point. Diverses
conditions de normalisation ont été proposées \cite{FeffermanGraham-ambmet,Gover-invtheory,LeeParker-Yamabeprob}.
Aux deux premiers ordres, elles coïncident
toutes avec la définition suivante:
\begin{defi} Une métrique $g$ est dite \emph{normale conforme à l'ordre $4$ en $p$} lorsque son tenseur
de Ricci vérifie:
\be\Ric_g(p) = 0\qquad\textrm{et}\qquad\left(\nabla_g\Ric_g\right)^\sym(p)=0\,,\tag{NC}\label{ricnormalconf}\ee
l'exposant $\sym$ désignant une symétrisation totale.
\label{defnormconf}\end{defi}
\noindent Étant fixé un point $p\in M$, on peut toujours trouver de telles métriques dans une classe conforme
fixée. En fait, cette condition laisse le $1$-jet de $g$ libre, et en fonction de celui-ci fixe son $3$-jet.

Avec cette normalisation, on obtient le résultat suivant:
\begin{theoIdvpG}[version partielle] Supposons $2k+1\leq n\leq 2k+3$ et que $\Pan_k^g$ admet une fonction
de Green $G_k^g$. Soit $p\in M$.
Si la métrique $g$ est normale conforme à l'ordre $4$ en $p$,
il existe un réel $A$ tel que pour tout $x$ au voisinage de $p$,
$$ G_k^g(p,x) = \frac{1}{\clap_{n,k}}r^{2k-n} + A + o(1)$$
où $r$ désigne la distance géodésique de $x$ à $p$ au sens de la métrique $g$, et
$$\clap_{n,k} =\Vol(\Sph^{n-1})2^{k-1}(k-1)!(n-2)(n-4)\dotsm(n-2k)\,.$$
\end{theoIdvpG}
\noindent L'énoncé complet fournit également un \da{} pour les autres valeurs de $k$; voir ci-après.
C'est le coefficient $A$ qui sera utilisé dans la suite, c'est pourquoi on se limite désormais
au cas $2k+1\leq n\leq 2k+3$ et $P_k^g$ inversible (cette condition est conformément invariante).

Deux métriques normales conformes en $p$ fournissent deux coefficients $A$ différents.
En raison de la covariance de $G_k$, on s'attend au lien suivant entre ceux-ci :
\begin{theoIcovconf} Soient $g$ et $\tilde{g}=\vphi^{\frac{4}{n-2k}}g$ deux métriques normales conformes
à l'ordre $4$ en $p$. Les nombres $A$ et $\tilde{A}$ produits par le théorème \ref{theodvpG}
pour $G_k^g$ et $G_k^{\tilde{g}}$ \resp{} vérifient
$$\tilde{A} =\vphi(p)^{-2}A\,.$$
\end{theoIcovconf}

Le coefficient $A$ s'étend ainsi en une densité conforme de poids $2k-n$:
\begin{defi} \label{defmassepk} Soit $g\in [g_0]$. Soit $p\in M$. Il existe une métrique
$\bar{g}=\psi^{\frac{4}{n-2k}} g$ normale conforme à l'ordre $4$ en $p$, pour
laquelle le théorème \ref{theodvpG} fournit un coefficient $\bar{A}$.
On appelle \emph{masse} en $p$ de l'opérateur $\Pan_k^g$ le nombre, indépendant de $\bar{g}$:
$$A_p^g =\psi(p)^{2}\bar{A}\,.$$
\end{defi}
\noindent Le terme \og{}densité conforme\fg{} signifie que si $\tilde{g} =\vphi^{\frac{4}{n-2k}}g$,
alors $A^{\tilde{g}}_p =\vphi(p)^{-2}A_p^g$. La fonction de $p$ ainsi obtenue s'avère lisse
(proposition \ref{propreg}). Sous réserve de positivité, elle permet de construire une métrique à la
manière de Habermann et Jost \cite{HabermannJost-methabjost, Habermann-LN}:
\begin{defi} \label{defmetHab}
La \emph{métrique de Habermann-Jost} est définie en tout point $p$ où la masse de l'opérateur $\Pan_k^g$
est positive par
$$\mathfrak{g}_p =\left(A_p^g\right)^{\frac{2}{n-2k}}g_p\,.$$
\end{defi}
\noindent Cette métrique est indépendante du choix de $g$ dans la classe conforme.

Dans le cas $k=1$, la positivité qui permet à Habermann de définir cette métrique est le théorème de la
masse positive de Schoen et Yau, appliqué au terme constant du \da{} de la fonction de Green de $\Pan_1$ par
Schoen \cite{Schoen-solYamabe}. Un résultat de Raulot et Humbert dans le cas conformément plat
\cite{RaulotHumbert-massepos} s'étend ici et traite le cas de l'opérateur de Paneitz $\Pan_2$.
\begin{theoImassepos}Si $M$ est de dimension $5$, $6$ ou $7$ et si la classe conforme $[g_{0}]$
vérifie:
\bi
\item l'invariant de Yamabe $Y([g_{0}])$ est strictement positif,
\item l'opérateur de Paneitz-Branson $\Pan_2^{g_0}$ est inversible,
\item et la fonction de Green $G_2^{g_0}$ est partout strictement positive,
\ei
alors pour tout $p\in M$ et toute métrique $g\in [g_{0}]$, la masse
de l'opérateur $\Pan_2^g$ vérifie
$$A_{p}^g\geq 0$$
et la nullité en un point implique que $(M,[g_{0}])$ est conformément équivalente à la sphère
$\Sph^n$ munie de la classe conforme de sa métrique ronde.
\end{theoImassepos}

Le terme de \og{}masse\fg{} employé ici provient du lien, observé par Schoen \cite{Schoen-solYamabe}, entre le coefficient
$A$ apparaissant pour la fonction de Green $G_1$ du laplacien con\-forme $\Pan_1$, et la masse ADM (pour
Arnowitt, Deser et Misner \cite{ADM-mass})
de la variété non compacte $\hat{M}$ munie de la métrique asymptotiquement plate $\hat{g}=\vphi^{4/(n-2)}g$,
où $\vphi$ est la fonction partielle $G_1(p,.)$, singulière en $p$.
Rappelons qu'une variété $(\hat{M},\hat{g})$ est dite asymptotiquement
plate lorsqu'il existe un système de coordonnées
$$z=(z^i)_{1\leq i \leq n}: \hat{M}-K \xrightarrow{\;\sim\;} \R^n-B$$
sur l'extérieur d'un compact $K$, à valeurs dans un voisinage de l'infini de $\R^n$, dans lequel les
coefficients de la métrique s'écrivent
$$\hat{g}_{ij}=\delta_{ij} + O(\norm{z}^{-\tau})$$
avec $\partial_i\hat{g}_{jk} = O(\norm{z}^{-\tau-1})$ et des décroissances similaires pour les dérivées successives.
La courbure scalaire s'écrit dans cette carte
$$\Scal^{\hat{g}} =\partial_j\left(\partial_i\hat{g}_{ij} - \partial_j\hat{g}_{ii}\right) +
O(\norm{z}^{-2\tau-2})\,.$$
Le terme dominant à l'infini $\norm{z}\to\infty$ est une divergence, dont l'intégrale sur $\hat{M}$, si
elle avait un sens, serait un terme de \og{}bord à l'infini\fg{}:
$$\lim_{R\to\infty}\oint_{\set{\norm{z}=R}}\left(\partial_i\hat{g}_{ij}
-\partial_j\hat{g}_{ii}\right) d\sigma^j\,.$$
Cette expression définit la masse ADM\footnote{La définition d'Arnowitt, Deser et Misner,
ainsi que d'autres références, introduisent un facteur de normalisation. On n'en emploie pas
ici.} de $\hat{g}$. Chru\'sciel \cite{Chrusciel-mass} et Bartnik
\cite{Bartnik-mass} ont prouvé que c'est un authentique invariant lorsque la courbure
scalaire est intégrable et l'ordre $\tau$ est assez grand. On peut tenter
de copier cette définition pour fabriquer des invariants asymptotiques à partir d'autres scalaires
riemanniens, par exemple les laplaciens itérés de la courbure scalaire. Si la méthode est générale
\cite{BMich-invas}, ici elle tourne court et aboutit aux
invariants suivants:
\begin{defi} Soit $k\geq 2$. Sur une variété riemannienne $(\hat{M},\hat{g})$, on suppose
que le laplacien itéré de la courbure
scalaire $\Delta_{\hat{g}}^{k-1}\Scal^{\hat{g}}$ est intégrable relativement à la mesure riemannienne
$\ud\vol_{\hat{g}}$. On nomme \emph{masse d'ordre $k$} l'intégrale
$$m_k(\hat{g}) =\int_{\hat{M}}\Delta_{\hat{g}}^{k-1}\Scal^{\hat{g}}\, \ud\vol_{\hat{g}}\,.$$
\label{defmasseas}\end{defi}
\noindent Ces masses asymptotiques sont reliées à la masse des opérateurs GJMS définie précédemment :
\begin{theoIinvas}
Si $2k+1\leq n\leq 2k+3$ et si $\Pan_k^g$ admet une fonction de Green $G_k^g$, soit $G$ une fonction
strictement positive coïncidant avec $G_k^g(p,.)$ au voisinage d'un point $p\in M$. On définit sur
$\hat{M} = M-\set{p}$ la métrique
$$\hat{g}=G^{\frac{4}{n-2k}} g\,.$$
Alors $\hat{g}$ admet une masse d'ordre $k$ reliée à la masse de l'opérateur $\Pan_k^g$ par
$$A_p^g=\frac{n-2k}{4(n-1)}m_k(\hat{g})\,.$$
\end{theoIinvas}
\noindent L'observation de Schoen est le cas $k=1$, pour lequel $m_1$ désigne la masse ADM.\\

Ce texte s'organise de la manière suivante. Le théorème \ref{theodvpG} calculant la singularité de
la fonction de Green est établi dans la première partie. L'outil essentiel est le lemme \ref{pkralpha}
estimant $\Pan_k^g r^{2k-n}$ en un point où la métrique $g$ vérifie
\eqref{ricnormalconf}. La preuve nécessite une connaissance plus précise de l'opérateur GJMS que
son seul premier terme, explicitée dans la proposition \ref{propsuitepk}.

La partie \ref{partiecovconf} traite de l'invariance conforme énoncée par le théorème \ref{theocovconf}. La
difficulté principale provient du comportement de la distance riemannienne lors d'un changement conforme de
métriques; elle est traitée par le lemme \ref{lemmerconf}. La régularité de la densité conforme introduite
dans la définition \ref{defmassepk} fait l'objet de la proposition \ref{propreg}.

Le théorème \ref{theomassepos} est démontré dans la partie \ref{secp2}. On y calcule également
explicitement la masse des opérateurs GJMS des classes conformes canoniques des espaces sphériques
(proposition \ref{propquosphere}), permettant
de voir que la métrique de Habermann-Jost y est bien définie et qu'elle n'est en général
pas proportionnelle à la métrique ronde. Ce calcul utilise la proposition \ref{propmasserevet} sur
le comportement de la masse sous un revêtement riemannien.

Enfin, la partie \ref{partieinvas} démontre le théorème \ref{theoinvas} reliant la masse de la définition
\ref{defmassepk} et la masse asymptotique de la définition \ref{defmasseas}.

\begin{notation}
On désigne par $M$ une variété différentielle de classe $C^\infty$ compacte et connexe. Pour une
métrique riemannienne $g$, la connexion de Levi-Civita sera notée  $\nabla^g$,
le tenseur de Riemann $\Rm^g$, le tenseur de Ricci $\Ric^g$, la courbure scalaire $\Scal^g$,
\emph{etc} de manière classique. On note $\ud\vol_g$ la mesure de volume riemannien.
Souvent l'exposant ou l'indice $g$ sera omis. S'il faut faire la distinction entre deux métriques $g$ et
$\tilde{g}$,
on pourra abréger $\nabla^{\tilde{g}}$ en $\tilde{\nabla}$, et de même pour les autres objets riemanniens
ou d'autres métriques $\bar{g}$, $\hat{g}$.

On utilisera le tenseur de Schouten
\begin{align*} \Rho & =\frac{1}{n-2}\left(\Ric -\frac{\Scal}{2(n-1)} g\right)\\
\intertext{et sa trace}
\Iota & =\tr_g\Rho =\frac{\Scal}{2(n-1)} \end{align*}

Les conventions de signes sont:
\bi
\item pour le tenseur de Riemann $\Rm(X,Y) =[\nabla_X,\nabla_Y]-\nabla_{[X,Y]}$,
\item pour le laplacien $\Delta = -\nabla^i\partial_i$.
\ei
\end{notation}

\section{Singularité des fonctions de Green des opérateurs GJMS} \label{partienoyauG}

L'objectif de cette partie est de préciser la singularité du noyau de Schwartz $K$ d'une paramétrix de $\Pan_k$.
L'idée générale est de résoudre \og{}en série entière\fg{} au voisinage de $p$ l'équation
$$\Pan_k^g K(p,.) =\delta_p$$
$\delta_p$ étant la masse de Dirac en $p$. En s'inspirant du cas euclidien (où $\Pan_k=\Delta^k$) on voit
aisément que
$$\Pan_k^g\left(\frac{1}{\clap_{n,k}} r^{2k-n}\right)=\delta_p +\text{reste}$$
où $r$ est la distance à $p$ induite par $g$. On cherche donc à exprimer $K$ sous la forme
$$K(p,.) =\frac{1}{\clap_{n,k}} r^{2k-n} +\psi$$
de sorte que $\Pan_k^g\psi$ compense le \og{}reste\fg{}.

Plus exactement, on veut par un bon choix de facteur conforme rendre ce \og{}reste\fg{} aussi petit que possible,
afin de trouver aisément $\psi$. C'est l'objet de la section \ref{secpkralpha}; c'est là qu'intervient la
normalisation \eqref{ricnormalconf}.
Il est nécessaire pour cela de calculer le terme dominant de $\Pan_k-\Delta^k$, ce qui est fait dans
la section \ref{secsuitepk}. La section \ref{secdvpG} en déduit le développement de la singularité à
l'ordre souhaité.

\subsection{Second terme des opérateurs GJMS} \label{secsuitepk}

\begin{prop} \label{propsuitepk} L'opérateur GJMS d'ordre $2k$ s'écrit
\begin{align*}
\Pan_k &=\Delta^k + 2\sum_{j=0}^{k-2} (j+1)(k-1-j)\Delta^j T\Delta^{k-2-j}\\
       &\quad -\sum_{j=0}^{k-1}\left(2j+2-k-\frac{n}{2}\right)\Delta^j\Iota\Delta^{k-1-j} +U\\
\intertext{où $T$ est l'opérateur}
Tv &= 2\Rho^{ij}\nabla^2_{ij}v + (\partial^i\Iota)\partial_iv
\end{align*}
et $U$ est un opérateur différentiel de degré au plus $2k-4$.
\end{prop}

\begin{proof}
On utilise la définition de \cite{GJMS} des opérateurs $\Pan_k$ à l'aide de la métrique ambiante
de Fefferman et Graham \cite{FeffermanGraham-confinv,FeffermanGraham-ambmet}. Rappelons-en brièvement les
propriétés qui seront pertinentes ici.

Si $g$ est une métrique sur $M$, on pose $\mathcal{G}\to M$ le
$\R^\times_+$-fibré principal des métriques de la classe conforme $[g]$. La métrique ambiante $\tilde{g}$ est
une métrique sur $\tilde{\mathcal{G}}=\mathcal{G}\times\R$ qui s'écrit dans des coordonnées locales adaptées
$(t,x^i,\rho)$:
\be\tilde{g} = t^2g_{\rho,ij}dx^idx^j+2\rho dt^2+2tdtd\rho \label{ambiantenormale} \ee
où $\rho$ est la coordonnée dans le facteur $\R$, $t$ une coordonnée homogène sur la fibre du facteur
$\mathcal{G}$, et $(x^i)$ un jeu local de coordonnées sur $M$. Ici $g_\rho = g_{\rho,ij} dx^idx^j$ est une famille
lisse à un paramètre de métriques sur $M$, telle que $g_0=g$.
Le jet de $\tilde{g}$ le long de $\mathcal{G} \simeq \mathcal{G}\times \set{0} \subset \tilde{\mathcal{G}}$ est
formellement déterminé à un ordre élevé (dépendant de $n$) par l'équation $\Ric^{\tilde{g}}=0$.

Si $u \in C^\infty(M)$ est prolongée arbitrairement en
$(x,\rho) \mapsto \tilde{u}(x,\rho)$ sur $M\times \R$,
$$\Pan_k^gu=\Delta_{\tilde{g}}^k(t^{k-\frac{n}{2}}\tilde{u})|_{\rho=0,t=1}$$
ne dépend que de la partie formellement déterminée du jet de $g$ pour tout $k$ si $n$ est impair, et pour $k\leq n/2$
si $n$ est pair; dans ces cas il est indépendant du choix de $\tilde{u}$.

Le laplacien se calcule aisément dans la forme \eqref{ambiantenormale} de la métrique ambiante $\tilde{g}$,
voir \cite{GJMS} qui diffère par la convention de signe pour le laplacien:
\[ \tilde{\Delta}  =  t^{-2}\left[ \Delta_{g_\rho} + 2 \rho \partial_\rho^2
    + \left(2 \rho \frac{\partial_\rho \sqrt{g_\rho}}{\sqrt{g_\rho}} - 2 w - n +2 \right) \partial_\rho
    -w \frac{\partial_\rho \sqrt{g_\rho}}{\sqrt{g_\rho}} \right] \]
sur l'ensemble des fonctions homogènes en $t$ de degré $w$. Ici $\sqrt{g_\rho}$ désigne la racine carrée
du déterminant de la matrice $(g_{\rho,ij})$.

Notons $\tilde{D}_w= 2 \rho \partial_\rho^2 + \left(2 \rho \frac{\partial_\rho \sqrt{g_\rho}}{\sqrt{g_\rho}}
- 2 w - n +2 \right) \partial_\rho -w \frac{\partial_\rho \sqrt{g_\rho}}{\sqrt{g_\rho}}$. Alors:
\begin{align} \Pan_k^gu & = \left.\left(\Delta_{g_\rho} + \tilde{D}_{k-2(k-1)-\frac{n}{2}}\right)
   \dotsm\left(\Delta_{g_\rho} + \tilde{D}_{k-2-\frac{n}{2}}\right)
   \left(\Delta_{g_\rho} + \tilde{D}_{k-\frac{n}{2}}\right)\tilde{u}\right\vert_{\rho=0} \notag \\
   & = \left[\Delta_{g_\rho}^k \tilde{u}
      +\sum_{j=0}^{k-1} \Delta_{g_\rho}^j \tilde{D}_{w_j} \Delta_{g_\rho}^{k-1-j} \tilde{u}
      +F(\Delta_{g_\rho},\tilde{D}) \tilde{u}\right]_{\rho=0} \label{dvpk1}
\end{align}
où $w_j=k-\frac{n}{2}-2(k-1-j)$, et $F(\Delta_{g_\rho},\tilde{D})$ est un polynôme non commutatif
en les $\tilde{D}_{w_j}$
et $\Delta_{g_\rho}$, de degré total $k$, et de degré partiel $k-2$ en $\Delta_{g_\rho}$.

Afin d'évaluer le membre de droite de \eqref{dvpk1}, on choisit $\tilde{u}$ indépendante de $\rho$. On observe que
$\Delta_{g_\rho}$ dérive dans les directions $\partial_{x^i}$ de $M$, et pas dans la direction $\partial_\rho$;
et inversement les $\tilde{D}_w$ ne dérivent que dans la direction $\partial_\rho$. Ainsi:
\bi
\item On a \be \Delta_{g_\rho}^k\tilde{u} \vert_{\rho=0} = \Delta_g^k
   \left(\tilde{u}\vert_{\rho=0}\right) = \Delta_g^k u\,. \label{dvpk2}\ee
\item Dans le second terme du membre de droite de \eqref{dvpk1}, les facteurs $\rho$ des $\tilde{D}_{w_j}$
disparaissent lors de l'évaluation en $\rho=0$:
\begin{subequations}
\begin{align} \Delta_{g_\rho}^j \tilde{D}_{w_j} \Delta_{g_\rho}^{k-1-j} \tilde{u} \vert_{\rho=0} & =
   (-2w_j -n+2) \Delta_g^j \partial_\rho \Delta_{g_\rho}^{k-1-j} \tilde{u}\vert_{\rho=0} \label{dvpk3}\\
   & \quad -\,w_j \Delta_{g}^j \frac{\partial_\rho \sqrt{g_\rho}}{\sqrt{g_\rho}} \Delta_g^{k-1-j}
       \tilde{u} \vert_{\rho=0}\,. \label{dvpk4}
\end{align} \end{subequations}
\bi
\item[$\circ$] On développe la partie droite de la ligne \eqref{dvpk3} à l'aide de la règle de Leibniz,
avec $\partial_\rho \tilde{u}=0$:
\[\Delta_g^j \partial_\rho \Delta_{g_\rho}^{k-1-j} \tilde{u}\vert_{\rho=0} = \Delta_g^{j}
\sum_{l=0}^{k-2-j}\Delta_g^l\left[\partial_\rho \Delta_{g_\rho}\right]_{\rho=0} \Delta_g^{k-2-j-l} u\,.\]
La dérivée du laplacien est (\emph{cf.} \cite{Besse}, 1.184)
\begin{align*} \left(\partial_\rho \Delta_{g_\rho}\right) v & = g_\rho\left(\partial_\rho g_\rho, \nabla^2_gv\right)
   - g_\rho\left(dv, \delta_g (\partial_\rho g_\rho) + \frac{1}{2} d \tr_g(\partial_\rho g_\rho)\right)\\
   & = 2g(\Rho,\nabla^2_gv) + g(dv,d\Iota) = Tv \text{ en } \rho=0
\end{align*}
car $\partial_\rho g_\rho \vert_{\rho=0} = 2 \Rho$, \emph{cf.} \cite{FeffermanGraham-ambmet}, et d'après
l'identité de Ricci contractée $\delta_g \Rho = -d\Iota$.
\item[$\circ$] Comme $\frac{\partial_\rho \sqrt{g_\rho}}{\sqrt{g_\rho}} =
  \frac{1}{2} \tr_{g_\rho} \left(\partial_\rho g_\rho\right)$,
le terme de la ligne \eqref{dvpk4} égale
$$ -w_j \Delta^j_g \Iota \Delta_g^{k-1-j}u\,. $$
\ei
Ainsi le deuxième terme du membre de droite de \eqref{dvpk1} vaut:
\begin{multline}
\sum_{j=0}^{k-1}\left[(-2w_j-n+2)\sum_{l=0}^{k-j-2}\Delta_g^{j+l}T\Delta_g^{k-2-j-l}
     -w_j\Delta_g^j\Iota\Delta_g^{k-1-j}\right]\\
= 2\sum_{j=0}^{k-1}(j+1)(k-j-1) \Delta_g^j T \Delta_g^{k-2-j}
- \sum_{j=0}^{k-1}(2j+2-k-\frac{n}{2}) \Delta_g^j \Iota \Delta_g^{k-1-j}\,. \label{dvpk5}
\end{multline}
\item Développant de même les dérivations par rapport à $\rho$ du terme $F(\Delta,\tilde{D})$
de \eqref{dvpk1} à l'aide de la règle de Leibniz, et évaluant en $\rho=0$, on obtient un polynôme non
commutatif de degré $k-2$ en $\Delta_g$ et les dérivées $\partial_\rho^j(\Delta_{g_\rho})|_{\rho=0}$,
qui sont tous des opérateurs de degré au plus $2$. Ce terme fournit
donc un opérateur différentiel $U$ de degré au plus $2k-4$.
\ei
La proposition \ref{propsuitepk} résulte de ce dernier constat, \eqref{dvpk2} et \eqref{dvpk5} substitués dans
\eqref{dvpk1}.\end{proof}

\subsection{Estimation de %
\texorpdfstring{$\Pan_k r^{2k-n}$}{Pk r2k-n}} \label{secpkralpha}

Soit un point $p \in M$. On se place en coordonnées géodésiques en $p$, notées $(x^i)$, et on
note $r = \sqrt{\sum x^i{}^2}$ la distance à $p$ induite par $g$.

On introduit les espaces 
$\Homg_\alpha$ des fonctions $C^\infty$ sur $M -\set{p}$ qui sont homogènes de degré $\alpha$ en les $x^i$
au voisinage de $p$.
On notera $\grandO_\alpha$  une somme finie
$\sum_{j \in \N} \vphi_j h_{\alpha+j}$
où $h_{\alpha+j} \in \Homg_{\alpha+j}$ et $\vphi_j\in C^\infty(M)$.
On notera $\grandO_{\alpha,\ln}$ une somme
$P_\alpha + Q_\alpha \ln r$
où $P_\alpha,Q_\alpha$ sont $\grandO_\alpha$.
Ces deux notations seront utilisées de la même manière que $O(r^\alpha)$ et $O(r^\alpha \ln r)$. Elles
présentent l'avantage que, si $D$ est un opérateur différentiel linéaire de degré $d$ à coefficients $C^\infty$
sur $M$,
$D\grandO_\alpha$ est $\grandO_{\alpha-d}$, et $D\grandO_{\alpha,\ln}$ est
$\grandO_{\alpha-d,\ln}$.

Dans cette section, on n'utilisera
la notation $O(r^j)$ que pour des fonctions $C^\infty$ sur $M$ toute entière; ainsi un produit
$O(r^j)\grandO_\alpha$ est $\grandO_{\alpha+j}$.
En particulier, en notant $\Delta_0=-\sum \partial_i^2$ le laplacien de coordonnées, on a
$$\Delta_g=\Delta_0 + (\delta^{ij}-g^{ij})\partial_i\partial_j + g^{ij}\Gamma_{ij}^k\partial_k\,,$$
où $\delta^{ij}-g^{ij} = O(r^2)$ et les coefficients de Christoffel $\Gamma_{ij}^k$ sont $O(r)$. Donc si
\begin{align}
f &\in \grandO_{\alpha} && \text{(\resp{} } f \in \grandO_{\alpha,\ln}\text{),}\notag\\
\intertext{alors}
\Delta_g f &= \Delta_0 f + \grandO_{\alpha} &&
\text{(\resp{} }\Delta_g f = \Delta_0 f+\grandO_{\alpha,\ln}\text{),} \notag\\
\intertext{(avec \resp{} $\Delta_0 f =\grandO_{\alpha-2}$ ou $=\grandO_{\alpha-2,\ln}$) et par une récurrence
élémentaire}
\Delta_g^kf&=\Delta_0^kf +\grandO_{\alpha-2k+2} && \text{(\resp{} }\Delta_g f = 
\Delta_0^k f+\grandO_{\alpha-2k+2,\ln}\text{)} \label{deltak}
\end{align}
(avec \resp{} $\Delta_0^kf =\grandO_{\alpha-2k}$ ou $=\grandO_{\alpha-2k,\ln}$).

Le laplacien est donc assez proche du laplacien de coordonnées. Sous les conditions \eqref{ricnormalconf},
il l'est encore plus lorsqu'il est appliqué à une fonctions de $r$ seul, et $\Pan_k$ ressemble également beaucoup
à $\Delta_0^k$. Seul le cas suivant est fondamental pour la suite:

\begin{lemm} \label{pkralpha} Si $g$ est normale conforme à l'ordre $4$ en $p$,
alors au sens des distributions sur $M$:
$$\Pan_k^gr^{2k-n} = \clap_{n,k}\delta_p + \grandO_{4-n}\,.$$
\end{lemm}

\begin{proof}

On utilise l'écriture de $\Pan_k$ donnée par
la proposition \ref{propsuitepk}.

L'opérateur $U$ étant de degré au plus $2k-4$, on a
\be Ur^{2k-n} = \grandO_{4-n}\,. \label{pknormalconf0} \ee
D'après l'identité de Bianchi contractée, les conditions \eqref{ricnormalconf} impliquent que\linebreak
$\Scal^g(0)=0$ et $d\Scal^g(p)=0$, donc $\Scal^g=O(r^2)$. Donc
\be \Delta^j \Iota \Delta^{k-j-1} r^{2k-n} =
\Delta^j\left(O(r^2) \grandO_{2j+2-n}\right) = \grandO_{4-n}\,. \label{pknormalconf1} \ee

Considérons maintenant $\Delta^l r^\alpha$.
Sur une fonction radiale $f=f(r)$, en dehors de l'origine $r=0$ on a:
\begin{align*} \Delta_gf & = -\partial_r^2f - \frac{n-1}{r}\partial_rf
      - \frac{\partial_r \sqrt{g}}{\sqrt{g}} \partial_rf\\
    & = \Delta_0f - \frac{\partial_r \sqrt{g}}{\sqrt{g}}\partial_r f
\end{align*}
(où $\sqrt{g}$ désigne la racine carrée du déterminant de la matrice $(g_{ij})$ dans les coordonnées normales
cartésiennes). Or d'après \cite[§5]{LeeParker-Yamabeprob}, les conditions \eqref{ricnormalconf} assurent que
$$\det(g_{ij}) = 1 + O(r^4)\,.$$
Donc $\frac{\partial_r \sqrt{g}}{\sqrt{g}} \in \grandO_3$. Donc
$$\Delta_gr^\alpha =\Delta_0r^\alpha + \grandO_{\alpha+2}\,.$$
Par récurrence, en utilisant \eqref{deltak} et le fait que $\Delta_0^lr^\alpha =  C_{l,\alpha}r^{\alpha-2l}$
si $\alpha > 2l-n$
(la valeur de la constante $C_{l,\alpha}$ est sans importance), on déduit
que
\be \Delta_g^l r^\alpha = \Delta_0^l r^\alpha + \grandO_{\alpha-2l+4}\label{lapnormalconf}\ee
si $\alpha-2l>-n-2$. Cette égalité est à comprendre au sens des distributions lorsque $\alpha \leq 2l-n$.

En particulier,
\be \Delta^jT\Delta^{k-j-2}r^{2k-n} = \Delta^j T\left(C_jr^{2j+4-n} + \grandO_{2j+8-n}\right)
\label{tralpha0} \ee
(la valeur de la constante $C_j$ est ici aussi sans importance). Comme $\Delta^jT$ est de degré au plus $2j+2$,
on a
\be \Delta^jT\grandO_{2j+8-n} = \grandO_{6-n}\,. \label{tralpha11} \ee
Rappelons que $T=2\Rho^{ab}\nabla^2_{ab} +(\partial^a\Iota)\partial_a$. De $\partial_a\Iota=O(r)$, on déduit
\be \Delta^j\left[(\partial^a\Iota) \partial_ar^{2j+4-n}\right] = \grandO_{4-n}\,. \label{tralpha12} \ee
De plus
\begin{align} \Rho^{ab}\nabla^2_{ab}r^\alpha & = \Rho^{ab}\partial_a\partial_br^\alpha
         - \Rho^{ab}\Gamma_{ab}^c\partial_cr^\alpha \notag \\
     & = \alpha r^{\alpha-2}\Rho^{ab}\delta_{ab} + \alpha(\alpha-2)r^{\alpha-4}\Rho_{ab}x^ax^b
         + \grandO_{\alpha+1}\,. \label{tralpha1}
\end{align}
où puisque $\Rho_{ab}=O(r)$ et $\Gamma_{ab}^c=O(r)$, on a
$\Rho^{ab}\Gamma_{ab}^c\partial_cr^\alpha=\grandO_{\alpha+1}$. On a aussi
\be \Rho^{ab}\delta_{ab} = \Iota + \Rho^{ab}(\delta_{ab}-g_{ab}) = O(r^2)\,. \label{tralpha2} \ee
De plus
$$ \Rho_{ab}x^ax^b = \Rho_{ab}(p)x^ax^b + \left(\nabla_c \Rho_{ab}\right)(p)x^cx^ax^b + O(r^4)\,.$$
La première condition de \eqref{ricnormalconf} impose $\Rho(p)=0$. La seconde, avec l'identité de Bianchi,
implique que $\left(\nabla\Rho\right)^\sym(p)=0$, donc que $\left(\nabla_c \Rho_{ab}\right)(p)x^cx^ax^b
\equiv 0$. Ainsi
\be \Rho_{ab}x^ax^b=O(r^4)\,.\label{tralpha3}\ee
Reportant \eqref{tralpha2} et \eqref{tralpha3} dans \eqref{tralpha1}, on obtient
$$\Rho^{ab}\nabla^2_{ab}r^\alpha = \grandO_{\alpha}$$
d'où
\be \Delta^j\left(\Rho^{ab}\nabla_{ab}^2r^{2j+4-n}\right) = \grandO_{4-n}\,.\label{tralpha13}\ee
Cette dernière observation \eqref{tralpha13}, avec \eqref{tralpha11} et \eqref{tralpha12}, reportés dans
\eqref{tralpha0}, permet de conclure que
\be \Delta^jT\Delta^{k-j-2}r^{2k-n} = \grandO_{4-n}\,.\label{snormalconf}\ee

D'après \eqref{lapnormalconf} spécialisé à $l=k$ et $\alpha=2k-n$, sur un voisinage de $p$
$$\Delta_g^kr^{2k-n} = \clap_{n,k}\delta_p + \grandO_{4-n}\,.$$
Utilisant alors \eqref{snormalconf}, \eqref{pknormalconf0} et \eqref{pknormalconf1}, avec la forme de
$\Pan_k$ donnée par la proposition \ref{propsuitepk}, on obtient le lemme \ref{pkralpha}.\end{proof}

\subsection{Singularité} \label{secdvpG}

Opérateur elliptique, $\Pan_k$  admet des paramétrix, qui en sont des inverses modulo
opérateur régularisant. 

\begin{theo} \label{theodvpG}
Si au point $p \in M$, la métrique $g$ vérifie \eqref{ricnormalconf}, alors
le noyau de Schwartz $K$ de toute paramétrix de $\Pan_k^g$ admet au voisinage de $p$
le développement asymptotique suivant:
\[ K(p,x) = \begin{cases} \clap_{n,k}^{-1}r^{2k-n} + A + o(1) & \text{si }
            2k+1 \leq n \leq 2k+3\\[2pt]
      \clap_{n,k}^{-1}r^{-4} + B \ln r + O(1) & \text{si }n=2k+4 \\[2pt]
      \clap_{n,k}^{-1}r^{2k-n} + O(r^{2k-n+4}) & \text{si } n \geq 2k+5 \,,
   \end{cases}\]
où $\clap_{n,k}$ est défini dans l'introduction.

Le $o(1)$ apparaissant dans ce développement est tel que $\partial^Jo(1) = o(r^{-\abs{J}})$ pour tout
multi-indice $J$, et de même $\partial^JO(r^\alpha) = O(r^{\alpha-\abs{J}})$.
\end{theo}

\begin{rema}\label{rqdvpG}
\begin{enumerate}[(i)] 
\item \label{rqdvpG1} Comme expliqué au début de la partie, ce \da{} est une conséquence du fait que
$$\Pan_k^gr^{2k-n} = \clap_{n,k} \delta_p + O(r^{4-n})\,.$$
C'est pourquoi la normalisation conforme \eqref{ricnormalconf} est indispensable.
\item Il y a un \da{} analogue si $n=2k$:
$$K(p,x) = c \ln r + A + o(1)\,.$$
Le coefficient $A$ apparaissant lorsque $K$ est le noyau de Schwartz
du pseudo-inverse défini par
$$\Pan_k^gK=K\Pan_k^g = \Id - \Pi_{\ker \Pan_k}$$
a été étudié dans \cite{Steiner-thesis}. Il ne présente pas l'invariance conforme
du théorème \ref{theocovconf} (voir aussi la remarque
\hyperref[rqcovconf.noncov]{\ref*{rqcovconf}-(\ref*{rqcovconf.noncov})}).
\item Pour le cas $n=2k+4$, il est facile de voir que le coefficient $B$ de $\ln r$ est un invariant
conforme de poids $-2$, donc s'écrit $c_n \abs{W}^2$: voir \cite{Ponge-preprintlogsing}. Le coefficient
$c_n$ n'est pas nul en général: par exemple dans le cas du laplacien conforme en dimension
$6$, les résultats de Lee et Parker \cite[lemme 6.4]{LeeParker-Yamabeprob} montrent que $c_6 = 1/1440$.
\item La méthode utilisée permet en théorie de poursuivre le développement asymptotique lorsque
$n \geq 2k+5$. Néanmoins, cela demande de connaître les termes de degrés inférieurs de $\Pan_k$, et de contrôler
précisément la différence entre le laplacien de Beltrami et le laplacien de coordonnées géodésiques.
L'espoir d'obtenir un développement de la forme $K=\clap_{n,k}^{-1}r^{2k-n} + A + o(1)$ étant nul
(en raison, s'il en fallait une, du terme logarithmique en dimension $n=2k+4$), les calculs ne sont pas
faits ici.
\end{enumerate} \end{rema}

\begin{proof}[Démonstration du théorème \ref{theodvpG}]
La théorie des opérateurs pseudo-différentiels assure que pour tout $l$, en coordonnées quelconques $(x^i)$ au
voisinage de $p$, $K_p(x)=K(p,x)$ admet le développement asymptotique
\be K_p(x) = \sum_{j=2k-n}^{l} q_j(x) + \sum_{j=0}^{l}p_j(x)\ln \norm{x} + \vphi(x) \label{dvpK}\ee
où $q_j \in \Homg_j$, $p_j$ est au voisinage de $p$ un polynôme homogène en les $x^i$ et $\vphi \in C^l(M)$,
voir \cite[chapitre 7, proposition 2.8]{Taylor2}.

Supposons désormais que $g$ vérifie \eqref{ricnormalconf}, et que les $x^i$ sont des coordonnées géodésiques sur
un voisinage de $p$. On commence par calculer le noyau de $\Delta_0^k$, $\Delta_0$ étant toujours le laplacien
de coordonnées. Pour cela on fixe une petite boule 
$U$ centrée en $p$ sur laquelle on supposera que les éléments de $\Homg_\alpha$ sont effectivement
homogènes.

On se restreint à $U-\set{p}$. En coordonnées sphériques, sur les fonctions de $\Homg_{\alpha}$,
$$\Delta_0 = r^{-2}\left(\Delta_\Sph-\alpha(\alpha+n-2)\right)\,,$$
où $\Delta_\Sph$ désigne
le laplacien de la sphère ronde $\Sph^{n-1}$. Notons $c_\alpha = \alpha(\alpha+n-2)$. Les valeurs propres de
$\Delta_\Sph$ sont les $c_j$ avec $j \in \N$, ce qui assure que $\Delta_0$ est non-inversible sur 
$\Homg_\alpha$ \ssi{} $\alpha \in \N$ ou $\alpha \in 2-n-\N$.
Alors sur $\Homg_\alpha$,
$$ \Delta_0^k = r^{-2k}(\Delta_\Sph - c_{\alpha-2k+2})
      (\Delta_\Sph - c_{\alpha-2k+4}) \dotsm (\Delta_\Sph - c_{\alpha-2})
      (\Delta_\Sph - c_\alpha)$$
est inversible \ssi{} $\alpha \notin \Z$ ou $2k-n < \alpha < 0$. De plus, 
le noyau de $\Delta_\Sph$ étant constitué des fonctions constantes, celui de
$\Delta_0^k|_{\Homg_{2k-n}}$ (toujours en restriction à $U-\set{p}$) est engendré par $r^{2k-n}$, et
$\ker \Delta_0^k|_{\Homg_0}$ par les fonctions constantes.\\

Pour une paramétrix $K_p$ quelconque, il existe une fonction $u \in C^\infty(M)$ telle que
\be \Pan_kK_p=\delta_p + u\,. \label{eqK}\ee
      
Le développement \eqref{dvpK} implique que
$$K_p(x) = q_{2k-n}(x) + \grandO_{2k-n+1,\ln}\,,$$
donc que
$$\Pan_k K_p(x) = \Delta_0^kq_{2k-n}(x) + \grandO_{1-n,\ln}\,.$$
En restriction à $M-\set{p}$, $\Delta_0^k q_{2k-n} \in \Homg_{-n}$. Or d'après \eqref{eqK}, en restriction à
$M-\set{p}$, $\Pan_kK_p$ est $O(1)$ au voisinage de $p$, ce qui impose
$$\Delta_0^kq_{2k-n}|_{U-\set{p}} = 0\,.$$
D'après la remarque sur le noyau de $\Delta_0^k|_{\Homg_{2k-n}}$,
on a $q_{2k-n} = c r^{2k-n}$. Le lemme \ref{pkralpha} impose la valeur $c=\clap_{n,k}^{-1}$
pour que $\Pan_kq_{2k-n} =\delta_p + \grandO_{1-n,\ln}$.

Ainsi, utilisant encore le lemme \ref{pkralpha}
\begin{align*} \Pan_k\left(K_p-q_{2k-n}\right) & = \delta_p - \delta_p +\grandO_{4-n} \\
 & = \Pan_k\left(q_{2k-n+1} + \grandO_{2k-n+2,\ln}\right)\\
 & = \Delta_0^kq_{2k-n+1} + \grandO_{2-n,\ln}
\end{align*}
On a donc $\Delta_0^k q_{2k-n+1}\in \Homg_{1-n} \cap \grandO_{2-n,\ln} = \set{0}$.
Si $2k-n+1<0$, cela impose $q_{2k-n+1}=0$ d'après les remarques sur
le noyau de $\Delta_0^k$.
En poursuivant de même, on obtient
\bi \item si $2k-n+2<0$, $q_{2k-n+2}=0$,
\item si $2k-n+3<0$, $q_{2k-n+3} = 0$. \ei
Dans le cas $n\geq 2k+4$, on a donc dans \eqref{dvpK}
$$K_p=\clap_{n,k}^{-1}r^{2k-n} + \vphi + \grandO_{2k-n+4,\ln}$$
avec $\vphi \in C^l(M)$.
Ceci prouve les cas $n=2k+4$ et $n\geq 2k+5$  du théorème, dans lequel le $O(r^{2k-n+4})$
est le $\vphi + \grandO_{2k-n+4,\ln}$ ci-dessus. L'ordre $l$ du \da{} \eqref{dvpK} étant arbitraire, on
obtient les dominations annoncées des $\partial^JO(r^{2k-n+4})$.

Dans le cas $2k+1 \leq n \leq 2k+3$, on aboutit au développement
$$K_p = \clap_{n,k}^{-1}r^{2k-n}+B \ln r + q_0 +\grandO_{1,\ln} + \vphi\,.$$
Comme
$$\Delta_0^k\ln r = -2^{k-1}(k-1)!(n-2)(n-4)\dotsm(n-2k)r^{-2k}$$ avec $-2k < 4-n$,
le même raisonnement que ci-dessus force $B=0$. De même, comme $\Delta_0^k q_0 \in \Homg_{-2k}$ avec
$-2k<4-n$, la fonction $q_0$ doit être dans le noyau de $\Delta_0^k|_{\Homg_0}$, donc être constante.
Le coefficient constant $A$ apparaissant
dans le développement est alors $q_0+\vphi(p)$. Le $o(1)$ est la suite du développement \eqref{dvpK} et vaut
$$\sum_{j=1}^l(q_j(x) + p_j(x)\ln r) + (\vphi(x)-\vphi(p)) \in \grandO_{1,\ln} + C^l\,.$$
L'ordre $l$ du développement \eqref{dvpK} étant
arbitraire, ceci fournit les dominations annoncées des $\partial^Jo(1)$.
\end{proof}

Il est à noter que cette méthode formelle permettrait, en poursuivant les calculs, de déterminer
toute la partie singulière de $K_p$, \cad{} dans \eqref{dvpK} les $p_j$ et les $q_j$
modulo fonction $C^\infty$. Ceux-ci sont donc fonctions des coefficients de l'opérateur $\Pan_k^g$, et
sont donc des invariants riemanniens locaux. Ils peuvent d'ailleurs s'obtenir en inversant le symbole
total de $\Pan_k^g$ au sens du calcul pseudo-différentiel, \emph{cf.} \cite{Taylor2}. À l'inverse,
le coefficient $A$ apparaissant dans le cas $2k+1 \leq n \leq 2k+3$ ne peut être calculé de cette manière,
et dépend de la paramétrix.

\section{Masse et densité conforme} \label{partiecovconf}

On se restreint désormais au cas $2k+1\leq n \leq 2k+3$ et $\Pan_k^g$ inversible. Sa fonction de
Green $G_k^g$ est le noyau intégral de son inverse. Le théorème \ref{theocovconf} assure que
lorsque $g$ est normale conforme à l'ordre $4$, la grandeur suivante est bien définie:
$$ A_p^g = \lim_{r\to 0} \left(G_k^g(p,x) - \clap_{n,k}^{-1}r^{2k-n}\right)$$
où $r$ est la distance géodésique de $x$ à $p$.

La section \ref{seccovconf} prouve la covariance conforme annoncée en introduction:
\begin{theo} \label{theocovconf} Si $g$ et $\tilde{g}=\vphi^{\frac{4}{n-2k}}g$ sont normales conformes
à l'ordre $4$ en $p$ (conditions \eqref{ricnormalconf}), alors
$$A_p^{\tilde{g}} = \vphi(p)^{-2}A_p^g\,.$$
\end{theo}
\begin{rema}\label{rqcovconf}
\begin{enumerate}[(i)]
\item Ce lien est fortement suggéré par la loi de transformation de $G_k^g$ sous un changement conforme
de métriques, rappelée en introduction. La difficulté provient du mauvais comportement de $r$, que
l'on contrôle avec la normalisation conforme \eqref{ricnormalconf}. C'est la raison pour laquelle
ce résultat n'est pas valide si $g$ ou $\tilde{g}$ n'est pas normale conforme.
\item \label{rqcovconf.noncov} On peut également définir une fonction de Green lorsque $\Pan_k^g$ n'est pas
inversible.
En confondant opérateur et noyau intégral pour alléger la notation, elle vérifie
$$\Pan_k^g G_k^g = G_k^g\Pan_k^g = \Id - \Pi^g$$
où $\Pi^g$ est la projection orthogonale, au sens $L^2$ induit par $g$, sur $\ker \Pan_k^g$. Cette projection
n'étant pas conformément covariante, $G_k^g$ ne l'est pas non plus. Pour expliciter, si $\tilde{g} =
\vphi^{\frac{4}{n-2k}}g$, alors la fonction
$$\tilde{G}(x,y) = \frac{G_k^g(x,y)}{\vphi(x)\vphi(y)}$$
est solution de
$$\Pan_k^{\tilde{g}} \tilde{G} = \tilde{G} \Pan_k^g = \Id - \vphi^{-1} \Pi^g \vphi\,.$$
L'opérateur $\vphi^{-1}\Pi^g\vphi$ est un projecteur sur $\ker \Pan_k^{\tilde{g}}$, mais pas
orthogonal au sens $L^2$ induit par $\tilde{g}$ (si $\vphi$ n'est pas constante).
On ne peut donc pas espérer étendre le théorème \ref{theocovconf} à ce cas.
\end{enumerate}
\end{rema}

La définition \ref{defmassepk} étend la définition de $A_p^g$ en une densité conforme de poids
$2k-n$, sans restriction sur $g$. La partie \ref{secreg} en établit la régularité:
\begin{prop} \label{propreg}
Pour toute métrique $g$, la fonction $p\mapsto A_p^g$ est de classe $C^\infty$.
\end{prop}

\subsection{Covariance conforme} \label{seccovconf}

On prouve ici le théorème \ref{theocovconf}.
Il s'agit de comparer $r$ et
$\tilde{r}$ dans le cas où $g$ et $\tilde{g} = \vphi^{\frac{4}{n-2k}}g$ sont toutes deux normales conformes
à l'ordre $4$ en $p$.
Cette comparaison est inspirée par une propriété
classique du groupe de Möbius, engendré par les isométries affines et les inversions polaires. En
dimension $3$ ou plus, un théorème de Liouville garantit que toute application conforme d'un ouvert
de $\R^n$ dans $\R^n$ est la restriction d'une transformation de Möbius.

\begin{prop} \label{proprconfplat}
Soit, pour  $n \geq 2$ et un ouvert $D$ de $\R^n$, une transformation de Möbius $h: D \to \R^n$:
si $\eucl$ désigne la métrique plate canonique, $h^*\eucl = \vphi^2 \eucl$. Alors
$$ \forall x,y \in D, \quad \abs{h(x)-h(y)}^2 = \vphi(x) \vphi(y) \abs{x-y}^2\,.$$
En particulier, si $0 \in D \subset \R^n$ est fixé par $h$,
$$ \forall x \in D, \quad d_{h^*\eucl}(x,0)^2 = d_\eucl(h(x),0)^2 = \vphi(0)\vphi(x) d_\eucl(x,0)^2\,.$$
\end{prop}

Elle est exposée par exemple par Beardon dans \cite{Beardon}. Comme elle est élémentaire mais cruciale,
en voici une courte démonstration.

\begin{proof}
De manière évidente, si cette propriété est vérifiée pour deux transformations $h_1$ et $h_2$, alors elle
l'est pour la composée $h_1 \circ h_2$. Or le groupe de Möbius est engendré par les isométries (affines) de
$\R^n$, les homothéties et une inversion, par exemple $x \mapsto x/\abs{x}^2$. Le lemme est immédiat dans les
deux premiers cas, où $\vphi$ est constante; il reste donc le cas de l'inversion. 

Celui-ci se ramène à la dimension $2$ car $h$ stabilise les
plans passant par l'origine: on considère $x,y \in \C$; la métrique est
$\eucl = \abs{dz}^2$ pour une coordonnée complexe $z$,
$h: z \mapsto 1/\bar{z}$, et $h^*\eucl = \abs{d (1/\bar{z})}^2=(\abs{z}^{-2})^2 \abs{d z}^2$,
donc le facteur conforme est $\vphi: z \mapsto \abs{z}^{-2}$. Ainsi
$$ \abs{h(y)-h(x)}^2 = \abs{\frac{1}{\bar{y}}-\frac{1}{\bar{x}}}^2
   = \abs{\frac{\bar{x}-\bar{y}}{\bar{x}\bar{y}}}^2=\vphi(x)\vphi(y)\abs{x-y}^2\,. $$
\end{proof}

La proposition \ref{proprconfplat} inspire et permet de prouver la comparaison asymptotique suivante:

\begin{lemm} \label{lemmerconf}
En dimension $3$ ou plus, si $g$ et $\tilde{g} = e^{2f}g$ vérifient \eqref{ricnormalconf} en $p$,
les distances à $p$ au sens de $g$ et de $\tilde{g}$
sont reliées par $$\tilde{r}^2 - e^{f+f(p)}r^2 = O(r^6)\,.$$
\end{lemm}

\begin{rema} 
L. Habermann a montré l'égalité asymptotique à l'ordre $4$ dans \cite{Habermann-LN}, sous l'hypothèse plus
faible que $\Ric_g(p)=\Ric_{\tilde{g}}(p)$.
\end{rema}

\begin{proof}[Démonstration du lemme \ref{lemmerconf}]
Quitte à redimensionner $\tilde{g}$ par une constante, on suppose $\tilde{g}_{p}=g_{p}$, soit $f(p)=0$. On
cherche le \dl{} de $r^2-e^{-f}\tilde{r}^2$ à l'ordre $5$ en $r=0$. Utilisant des coordonnées
$g$-géodésiques en $p$, on transporte le calcul sur un ouvert de $\R^n$, $p$ étant identifié à l'origine.

Soit $t\mapsto\sigma(t)= (\sigma^i(t))_{1\leq i\leq n} \in \R^n$ une $\tilde{g}$-géodésique issue de $0$,
parcourue à vitesse $1$ (pour $\tilde{g}$).
L'essentiel de la démonstration consiste à montrer que le \dl{} recherché a pour coefficients des polynômes
en $g(0)$, $g^{-1}(0)$, $\dot{\sigma}(0)$ et $d f(0)$, dont les coefficients ne dépendent que de $n$. Notons
$\mathscr{P}$ l'ensemble des fonctions de $t$ (à valeurs tensorielles) qui s'évaluent en $0$ en de tels
polynômes. L'espace $\mathscr{P}$ est stable par produit tensoriel et contractions, mais pas par
dérivation.\\

Commençons par montrer que $\nabla^2f \in \mathscr{P}$ et $\left(\nabla^3f\right)^\sym \in \mathscr{P}$
(on rappelle que $\nabla$ désigne la connexion de Levi-Civita de $g$, et $\tilde{\nabla}$ celle de
$\tilde{g}$).

Le 2-jet de $f$ est déterminé par
\[
0 = \wt{\Ric}(0)=\Big[\Ric + (n-2)d f \otimes d f - (n-2) \nabla^2 f
    -(n-2)\abs{d f}_{g}^2 g + (\Delta_g f) g\Big](0)
\]
D'où, comme $\Ric(0)=0$, $\Delta_g f (0) = \frac{n-2}{2}\abs{d f}_{g}^2(0)$ et, puisque $n>2$,
$$\nabla^2 f(0) = d f \otimes d f (0) - \frac{1}{2}\abs{d f}_{g}^2(0) g(0).$$
Par conséquent $\nabla^2 f \in \mathscr{P}$.

Le $3$-jet de $f$ est fixé par la condition $\left(\wt{\nabla}\wt{\Ric}\right)^\sym (0) = 0$.
En tout point
\[ \wt{\nabla}{\wt{\Ric}} = \nabla\left[\Ric + (n-2) d f \otimes d f - (n-2)\nabla^2f
      -(n-2)\abs{d f}_{g}^2 g + \Delta_g f\right]
    + \left(\wt{\nabla}-\nabla\right)(\wt{\Ric})
\]
Or on a
$$\left(\wt{\nabla}-\nabla\right)(\wt{\Ric})(0) = 0$$
car $\wt{\Ric}(0) = 0$ et $\tilde{\nabla}-\nabla$ est un opérateur différentiel de degré $0$. La nullité de
$\left(\wt{\nabla}\wt{\Ric}\right)^\sym(0)$ et de $\left(\nabla \Ric\right)^\sym(0)$ assure donc que
$$(n-2)(\nabla^3f)^\sym(0)+(d\Delta_g f \otimes g)^\sym(0)$$
est un polynôme  en $g^{\pm 1}(0)$, $d f (0)$ et
$\nabla^2f(0)$.\\
Soit $L:S^3T^*M \to S^3T^*M$ l'application linéaire telle que
\begin{align*} L(\nabla^3 f)^\sym & = (n-2)(\nabla^3f)^\sym + (d \Delta_g f \otimes g)^\sym\\
   & = (n-2) (\nabla^3f)^\sym - g^{ab}(\nabla_{k}\nabla_{a}\nabla_{b}f) g_{ij}
      (d x^{i} \otimes d x^{j}\otimes d x^{k})^\sym
\end{align*}
Elle est inversible si $n>2$ (ce qui assure que le $3$-jet de $f$ est bien déterminé par
$\nabla\Ric^{\sym}(0)$ et $\wt{\nabla}\wt{\Ric}^{\sym}(0)$, \emph{cf.}
par exemple \cite[chapitre 8]{FeffermanGraham-ambmet}),
et son inverse est donc polynomial en $g^{\pm 1}$. Donc $(\nabla^3f)^\sym \in \mathscr{P}$.\\

On étudie maintenant le \dl{} de $r(t)^2=\sum_{i}\sigma^{i}(t)^2$. Ses coefficients sont combinaisons
linéaires des $\sum_{i}\sigma^{i\,(k)}(0)\sigma^{j\,(l)}(0) = g(\sigma^{(k)},\sigma^{(l)})(0)$ pour
$k+l\leq 5$. Un tel terme est nul si $k=0$ ou $l=0$, on s'intéresse donc aux $\sigma^{(j)}$ pour $j\leq 4$.

{\it\'Etude de $\ddot{\sigma}$:}
\begin{align}
\nabla_{\dot{\sigma}}\dot{\sigma} & = \wt{\nabla}_{\dot{\sigma}}\dot{\sigma}
      -2d f (\dot{\sigma})\dot{\sigma} + \abs{\dot{\sigma}}_{g}^2 \nabla f \notag\\
   & = 0 -2d f (\dot{\sigma})\dot{\sigma} + \abs{\dot{\sigma}}_{g}^2 \nabla f  \label{spp1}\\
   & = \ddot{\sigma} + \G^{a}_{ij}\dot{\sigma}^{i}\dot{\sigma}^{j}\partial_{a} \label{spp2}
\end{align}
La ligne \eqref{spp1} assure que $\nabla_{\dot{\sigma}}\dot{\sigma}
\in \mathscr{P}$; et comme $\G_{ij}^{a}(0)=0$, \eqref{spp2} implique alors que $\ddot{\sigma}\in \mathscr{P}$.

{\it\'Etude de $\dddot{\sigma}$:}\\
Prenant la dérivation covariante de \eqref{spp1} dans la direction de $\dot{\sigma}$:
\begin{equation}
\nabla_{\dot{\sigma}}\nabla_{\dot{\sigma}}\dot{\sigma} =
      -2g(\nabla_{\dot{\sigma}}\nabla f, \dot{\sigma})\dot{\sigma}
      -2g(\nabla f, \nabla_{\dot{\sigma}}\dot{\sigma})\dot{\sigma}
      -2 g(\nabla f,\dot{\sigma}) \nabla_{\dot{\sigma}}\dot{\sigma}
   + 2 g(\dot{\sigma}, \nabla_{\dot{\sigma}}\dot{\sigma})\nabla f
      + \abs{\dot{\sigma}}_{g}^2 \nabla_{\dot{\sigma}}\nabla f \label{sppp1}
\end{equation}
Comme $\nabla_{\dot{\sigma}}\dot{\sigma} \in \mathscr{P}$ et $\nabla_{\dot{\sigma}}\nabla f
=g^{-1}(\nabla^2f(\dot{\sigma},\,.\,)) \in \mathscr{P}$, on déduit $(\nabla_{\dot{\sigma}})^2
\dot{\sigma} \in \mathscr{P}$.\\
Le dérivation covariante de \eqref{spp2} suivant $\dot{\sigma}$ s'écrit
\be \nabla_{\dot{\sigma}}\nabla_{\dot{\sigma}}\dot{\sigma} =
      \dddot{\sigma} + \partial_{k}\G_{ij}^{a}\dot{\sigma}^{i}\dot{\sigma}^{j}\dot{\sigma}^{k}
      \partial_{a} + 3 \G_{ij}^{a}\ddot{\sigma}^{i}\dot{\sigma}^{j}\partial_{a}
   +\mathscr{Q}(\G,\dot{\sigma}) \label{sppp2} \ee
où $\mathscr{Q}(\G,\dot{\sigma})$ est une expression quadratique en les $\G_{ab}^c$, donc nulle à l'ordre
$2$ en $0$. Le coefficient $\partial_{k}\G_{ij}^{a}(0)$ est combinaison linéaire des
$\partial_{i}\partial_{j}g_{kl}(0)$, eux-mêmes combinaisons linéaires des $\Rm_{ijkl}(0)$ d'après le \dl{}
classique de la métrique en coordonnées normales. Donc le terme
$\partial_{k}\G_{ij}^{a}\dot{\sigma}^{i}\dot{\sigma}^{j}\dot{\sigma}^{k}\partial_a$ s'évalue
en $t=0$, à facteur multiplicatif près, en $\Rm(\dot{\sigma},\dot{\sigma},\dot{\sigma},\,.\,)(0)=0$.\\
Donc finalement, au même titre que $(\nabla_{\dot{\sigma}})^2\dot{\sigma}$,
$\dddot{\sigma} \in \mathscr{P}$.

{\it\'Etude de $\sigma^{(4)}$:}\\
La dérivée covariante de \eqref{sppp1} suivant $\dot{\sigma}$ fait apparaître
$(\nabla_{\dot{\sigma}})^3\dot{\sigma}$
comme combinaison linéaire, à coefficients entiers, des termes suivants:
\bi
\itbul $g\left(\nabla_{\dot{\sigma}} \nabla f,\nabla_{\dot{\sigma}}\dot{\sigma}\right)\dot{\sigma}$.
D'après les résultats précédents sur $\nabla^2f$ et $\nabla_{\dot{\sigma}}\dot{\sigma}$, ce terme est dans
$\mathscr{P}$.
\itbul $g\left(\nabla_{\dot{\sigma}}\nabla f,\dot{\sigma}\right) \nabla_{\dot{\sigma}}\dot{\sigma}$. De même
ce terme est dans $\mathscr{P}$.
\itbul $g\left(\nabla f, \nabla_{\dot{\sigma}}\dot{\sigma}\right) \nabla_{\dot{\sigma}}\dot{\sigma}
   \in \mathscr{P}$ de même.
\itbul $\abs{\nabla_{\dot{\sigma}}\dot{\sigma}}_g^2 \nabla f \in \mathscr{P}$ de même.
\itbul $g\left(\dot{\sigma},\nabla_{\dot{\sigma}}\dot{\sigma}\right)\nabla_{\dot{\sigma}}\nabla f \in
   \mathscr{P}$ de même.
\itbul $g\left(\nabla f,(\nabla_{\dot{\sigma}})^2\dot{\sigma}\right)\dot{\sigma}$. Utilisant de plus les
résultats précédents sur $(\nabla_{\dot{\sigma}})^2\dot{\sigma}$, on constate que ce terme est dans
   $\mathscr{P}$.
\itbul $g\left(\nabla f,\dot{\sigma}\right)(\nabla_{\dot{\sigma}})^2\dot{\sigma} \in \mathscr{P}$ de même.
\itbul $g\left(\dot{\sigma}, (\nabla_{\dot{\sigma}})^2\dot{\sigma}\right) \nabla f
   \in \mathscr{P}$ de même.
\itbul $g\left(\nabla_{\dot{\sigma}}\nabla_{\dot{\sigma}} \nabla f,\dot{\sigma}\right)\dot{\sigma} =
   \left[\nabla^3f(\dot{\sigma},\dot{\sigma},\dot{\sigma}) - \nabla^2 f(\dot{\sigma},
   \nabla_{\dot{\sigma}}\dot{\sigma})\right] \dot{\sigma}$. Comme $(\nabla^3f)^\sym \in \mathscr{P}$, d'après
   les résultats précédents, ce terme est aussi dans $\mathscr{P}$.
\itbul $\abs{\dot{\sigma}}_g^2 \nabla_{\dot{\sigma}} \nabla_{\dot{\sigma}} \nabla f$. On ne connaît que
   $(\nabla^3f)^\sym(0)$, donc il n'est pas possible de dire si ce terme est dans $\mathscr{P}$. En revanche,
   de même que ci-dessus,
   $$g\left(\abs{\dot{\sigma}}_g^2 \nabla_{\dot{\sigma}} \nabla_{\dot{\sigma}} \nabla f,
   \dot{\sigma}\right) \in \mathscr{P}\,.$$
\ei
Par conséquent, $g((\nabla_{\dot{\sigma}})^3\dot{\sigma},\dot{\sigma}) \in \mathscr{P}$; c'est ce qui importe
pour la suite.\\
De manière analogue, on peut lister les termes dont la dérivée covariante de \eqref{sppp2} suivant
$\dot{\sigma}$ est combinaison linéaire:
\bi
\itbul des termes dont les symboles de Christoffel $\G_{ab}^c$ sont en facteur (en particulier
$\nabla_{\dot{\sigma}} \mathscr{Q}(\G,\dot{\sigma})$). Ils s'annulent en $0$.
\itbul $\partial_{k}\partial_{l}\G_{ij}^{a}\dot{\sigma}^{i}\dot{\sigma}^{j}\dot{\sigma}^{k}
   \dot{\sigma}^l\partial_{a}$. En $0$, ce sont des combinaisons linéaires des dérivées secondes des
   $g_{ij}$. Or, d'après le développement limité de la métrique en coordonnées normales, calculé par exemple
   dans \cite{LeeParker-Yamabeprob}, celles-ci sont combinaisons linéaires des $\nabla_i \Rm_{jkl}\haut{a}(0)$.
   Donc en $t=0$ ces termes sont multiples de
   $\nabla \Rm(\dot{\sigma},\dot{\sigma},\dot{\sigma},\dot{\sigma},.)(0) = 0$.
\itbul $\partial_{k}\Gamma_{ij}^{a}\dot{\sigma}^{k}\dot{\sigma}^{i}\ddot{\sigma}^j\partial_{a}$,
   avec $\ddot{\sigma}$ lui-même combinaison linéaire de $\dot{\sigma}$ et $\nabla f$. Comme plus haut,
   en $t=0$ ce terme vaut
   $$A \Rm_{\dot{\sigma},\dot{\sigma}}\nabla f(0) + B \Rm_{\dot{\sigma}, \nabla f}\dot{\sigma}(0)
   = B \Rm_{\dot{\sigma}, \nabla f}\dot{\sigma}(0)$$
   (avec $A$ et $B$ réels). Ceci est orthogonal à $\dot{\sigma}(0)$.
\ei
Utilisant \eqref{sppp2} dérivée suivant $\dot{\sigma}$ et le fait que
$g((\nabla_{\dot{\sigma}})^3 \dot{\sigma},\dot{\sigma}) \in \mathscr{P}$, on déduit que
$g(\sigma^{(4)},\dot{\sigma}) \in \mathscr{P}$.

En conséquence de quoi, $r^2$ admet en $t=0$ un \dl{} dont les coefficients sont polynomiaux en $g^{\pm 1}(0)$,
$\dot{\sigma}(0)$ et $d f (0)$, avec des coefficients ne dépendant que de $n$.\\

Pour terminer, comme $\tilde{r}=t$, le \dl{} à l'ordre 5 de $e^{-f}\tilde{r}^2$ est polynomial en $\dot{f}(0)$, 
$\ddot{f}(0)$ et $\dddot{f}(0)$, où pour simplifier $\dot{f}$, $\ddot{f}$, \emph{etc.} désignent les dérivées
de $f(\sigma(t))$ par rapport à $t$.
\bi
\itbul $\dot{f} = d f (\dot{\sigma}) \in \mathscr{P}$;
\itbul à des termes en $\ddot{\sigma}$, $\dot{\sigma}$ et $d f$ près, $\ddot{f}$ vaut
   $\nabla^2f(\dot{\sigma},\dot{\sigma})$ donc $\ddot{f} \in \mathscr{P}$;
\itbul et de même, modulo des termes en les dérivées d'ordre au plus $3$ de $\sigma$, et au plus $2$ de $f$,
   $\dddot{f}$ vaut $\nabla^3f(\dot{\sigma},\dot{\sigma},\dot{\sigma})$, donc est dans $\mathscr{P}$.\\
\ei

Pour conclure, tous les coefficients du \dl{} à l'ordre 5 de $r^2-e^{-f}\tilde{r}^2$ sont polynomiaux en 
$g^{\pm 1}(0)$, $d f(0)$ et $\dot{\sigma}(0)$, avec des coefficients qui ne dépendent que de $n$. On les calcule
donc dans un cas simple, le cas conformément plat. Pour cela, on prend pour $g$ et $\tilde{g}$ deux métriques
plates. Il existe alors une application de Möbius $h$, définie au moins au voisinage de l'origine et fixant
celle-ci, telle que $\tilde{g}=h^*g$. D'après la proposition \ref{proprconfplat}, on a alors l'égalité exacte
$\tilde{r}^2 = e^{f+f(0)}r^2$. Cela assure la nullité des coefficients du \dl{}. \end{proof}

Le lemme \ref{lemmerconf} implique, si $g$ et $\tilde{g}=\vphi^{\frac{4}{n-2k}}g$ vérifient toutes deux
\eqref{ricnormalconf} en $p$, que
$$\frac{\tilde{r}(x)^{n-2k}}{\vphi(p) \vphi(x) r(x)^{n-2k}} = 1+O(r^4)$$
donc si $n-2k \leq 3$
$$\frac{1}{\tilde{r}(x)^{n-2k}} - \frac{1}{\vphi(p) \vphi(x) r(x)^{n-2k}} = o(1)\,.$$
D'où
\begin{align*} A_p^{\tilde{g}} & = \lim_{x\to p}\left(\tilde{G}_k(p,x)-\frac{1}{\clap_{n,k}
         \tilde{r}(x)^{n-2k}}\right) \\
   & = \lim_{x \to p} \frac{1}{\vphi(x)\vphi(p)} \left(G_k(p,x)-\frac{1}{\clap_{n,k}
         r(x)^{n-2k}}\right)\\
   & =\vphi(p)^{-2} A_p^g\,.
\end{align*}
C'est ce qu'énonce le théorème \ref{theocovconf}.

\subsection{Régularité} \label{secreg}

On démontre ici la proposition \ref{propreg}.

On fixe $g$ dans $[g_0]$. Pour chaque point $p$, il existe une fonction $\psi_p$ telle que
$\psi_p^{\frac{4}{n-2k}}g$ vérifie \eqref{ricnormalconf} en $p$. En fait, la condition \eqref{ricnormalconf}
détermine le $3$-jet de $\psi_p$ en fonction polynomiale de son $1$-jet, \emph{cf.} \cite[§5]{LeeParker-Yamabeprob}.
Ainsi, il existe une fonction
$\psi \in C^\infty(M\times M)$ telle que, si l'on note $\psi_p$ la fonction partielle
$x \mapsto \psi(p,x)$, pour tout point $p$ la métrique $\bas{p}g=\psi_p^{\frac{4}{n-2k}} g$ est normale
conforme à l'ordre $4$ en $p$.

D'après \cite[chapitre 7, proposition 2.8]{Taylor2}, en coordonnées au voisinage de la diagonale de $M$ on a
$$G_k^g(x,y) = \sum_{j=2k-n}^{l} q_j(x,y-x) + \sum_{j=0}^{l}p_j(x,y-x)\ln \norm{x-y} + \vphi(x,y-x)$$
où les $\left(q_j(x,.)\right)_x$ sont des familles lisses de distributions homogènes de degré
$j$, $\left(p_j(x,.)\right)_x$ des familles lisses de polynômes de homogènes, et $\left(\vphi(x,.)\right)_x$
une famille lisse de fonctions $C^l$.
La fonction de Green $\bas{p}G$ de $\Pan_k^{(\bas{p}g)}$ est reliée à $G_k^g$ par
$$\bas{p}G(x,y) = \frac{G_k^g(x,y)}{\psi_p(x) \psi_p(y)}\,.$$
Elle admet donc un \da{} de la même forme que celui de $G_k^g$, qui
dépend de manière lisse de $p$:
$$\bas{p}G(x,y) = \sum_{j=2k-n}^{l} \bas{p}q_j(x,y-x) + \sum_{j=0}^{l}
   \bas{p}p_j(x,y-x)\ln r + \bas{p}\vphi(x,y-x)\,.$$
Ce développement coïncide avec \eqref{dvpK} lorsque $x$ est fixé égal à $p$.
Donc $$A_p^{(\bas{p}g)} = \bas{p}q_0(p,0) +\bas{p}\vphi(p,0)\,.$$
\bi
\item La fonction $p\mapsto\bas{p}\vphi(p,0)$ est $C^\infty$.
\item Pour pallier le fait que $q(0)$ n'est pas défini en général pour une fonction homogène $q$ de degré $0$,
on peut écrire
$$\bas{p}q_0(p,0)=\left.\frac{1}{\Vol(\Sph^{n-1})}\oint_{\Sph^{n-1}}\bas{p}
   q_0(x,\theta) d\theta\right|_{x=p}\,.$$
Cela assure la régularité de $p \mapsto \bas{p}q_0(p,0)$. \ei
Donc $A_p^{(\bas{p}g)}$ est fonction lisse de $p$; donc $A_p^g$,
qui vaut $\psi(p,p)^2A_p^{(\bas{p}g)}$, l'est également.

\section{Résultats de masse positive}

\subsection{Métrique de Habermann-Jost de %
\texorpdfstring{$\Pan_2$}{P2}} \label{secp2}

On se restreint dans cette partie à l'opérateur de Paneitz-Branson $\Pan_2$, en dimension $5$, $6$ ou $7$.
Il est explicitement donné par la formule
\be \Pan_2^{g} u = \Delta_{g}^2u + \delta^g\left(T_{g} d u\right) + \frac{n-4}{2}Q_2^{g} \label{defP2}\ee
avec
\[
T_{g} = (n-2) \Iota - 4 \Rho = \frac{(n-2)^2 +4}{2(n-1)(n-2)}\Scal^{g}g -\frac{4}{n-2} \Ric^{g}
\]
et
\begin{align*}
Q_2^{g} & = \Delta_g\Iota +\frac{n}{2}\Iota^2-2\abs{\Rho}^2\\
   & = \frac{1}{2(n-1)}\Delta_{g} \Scal^{g} + \frac{(n^3-4n^2 + 16n-16)}{2(n-1)^2(n-2)^2} \Scal^{g\,2}
   - \frac{2}{(n-2)^2} \abs{\Ric^{g}}^2\,.
\end{align*}

Toujours dans le cas où l'opérateur $\Pan_2$ est inversible dans la classe conforme $[g_0]$,
on ajoute les hypothèses suivantes:
\begin{enumerate}[(h1)]
\item\label{h1} l'invariant de Yamabe de $(M,[g_0])$ est strictement positif,
\item\label{h2} la fonction de Green $G_2^{g_0}$ est strictement positive.
\end{enumerate}
L'hypothèse (\hyperref[h1]{h\ref*{h1}}) n'est pas restrictive dans le but de construire une métrique
canonique dans une classe
conforme, puisque dans le cas Yamabe-négatif la métrique à courbure scalaire constante fait l'affaire.
En revanche l'hypothèse technique (\hyperref[h2]{h\ref*{h2}}) équivaut au principe du maximum pour
$\Pan_2$, qui n'est
pas acquis en général. Pour des résultats sur le signe de la fonction de Green de l'opérateur biharmonique
on peut se reporter aux travaux de Grunau et ses collaborateurs, entre autres \cite{GrunauRobert-biharmGreen}.

Dans le cadre localement conformément plat, la masse $A_p^g$ de $\Pan_2^g$ en $p \in M$ peut être définie en
toute dimension supérieure ou égale à $5$, à l'aide d'une métrique $g$ plate au voisinage de $p$. Raulot
et Humbert \cite{RaulotHumbert-massepos} ont montré, si $\Pan_2$ est un opérateur défini positif et sous
les hypothèses (\hyperref[h1]{h\ref*{h1}}) et (\hyperref[h2]{h\ref*{h2}}),
que $A_p^g\geq0$, avec égalité en un point seulement si
$(M,[g_0])$ est conformément équivalente à la sphère $\Sph^n$ munie de sa classe conforme standard. Leur
argument s'étend à la masse définie par la définition \ref{defmassepk}.

\begin{theo} \label{theomassepos} Si $M$ est de dimension $5$, $6$ ou $7$, si l'opérateur
$\Pan_2$ est inversible dans la classe conforme $[g_0]$, et sous les hypothèses
(\hyperref[h1]{h\ref*{h1}}) et (\hyperref[h2]{h\ref*{h2}}) ci-dessus,
alors pour tout $p \in M$ et toute métrique $g \in [g_{0}]$, la masse
de l'opérateur $\Pan_2^g$ définie par la définition \ref{defmassepk} vérifie:
$$A_{p}^g \geq 0$$
et la nullité (en un point) implique que $(M,[g_{0}])$ est conformément équivalente à la sphère $\Sph^n$
munie de la classe conforme de sa métrique ronde.
\end{theo}

\begin{proof}
Pour alléger, comme il n'y a pas d'ambiguïté, on omet l'indice $p$. On utilise une métrique
asymptotiquement plate $\hat{g}$ obtenue à partir de $g$ par \og{}projection stéréographique\fg{}
(comme il est d'usage de la nommer). On distinguera les opérateurs et invariants associés
à $\hat{g}$ de ceux de $g$ par un chapeau.

Le signe de $Y([g_0])$ assure que le laplacien conforme
$$\Pan_1=\Delta + \frac{n-2}{4(n-1)}\Scal$$
admet une fonction de Green $G_1$ strictement positive. Posant
$H(x)=\clap_{n,1} G_1(p,x)$ pour $x \in M-\set{p}$,
le théorème \ref{theodvpG} (qui n'est ici qu'un avatar de \cite[lemme 6.4]{LeeParker-Yamabeprob}),
assure que
$$ H=\begin{cases} \frac{1}{r^{n-2}} + \grandO_{6-n,\ln} & \text{si } n\geq 6\\[2pt]
                   \frac{1}{r^3} + a + \grandO_{1,\ln} & \text{si } n=5\,. \end{cases}
$$
ce qui se condense en
\be H = \frac{1}{r^{n-2}}\big(1+ar^3 + \grandO_{4,\ln}\big) \label{daH}\ee
avec $a=0$ si $n > 5$.
On considère sur $\hat{M}=M-\set{p}$ la métrique conforme à $g$ définie par
$$\hat{g}  = H^{\frac{4}{n-2}}g = \vphi^{\frac{4}{n-4}} g \quad
   \textrm{avec }\vphi = H^{\frac{n-4}{n-2}}\,.$$

Soit $(x^i)$ un système de coordonnées sur $M$, $g$-géodésiques en $p$. On définit les coordonnées
inversées par
$$z^i=\frac{x^i}{r^2}\,.$$
Par un calcul classique
\be\partial_{z^i}=r^{2}\partial_{x^i}-2x^ix^k\partial_{x^k}\,.\label{dxdz}\ee
Comme
$g(\partial_{x^i},\partial_{x^j})x^j=x^i$ (lemme de Gauss pour les coordonnées géodésiques), on a
\be g(\partial_{z^i},\partial_{z^j}) = r^{4}g(\partial_{x^i},\partial_{x^j})\,. \label{gz} \ee
En coordonnées géodésiques,
$g(\partial_{x^i},\partial_{x^j}) = \delta_{ij}+\grandO_{2}$,
donc d'après \eqref{daH}:
\be\hat{g}(\partial_{z^i},\partial_{z^j}) = \delta_{ij} + \grandO_{2}
   + \grandO_{3,\ln}\label{asplatr}\ee
(où le $\grandO_{2}$ et le $\grandO_{3,\ln}$ sont toujours à entendre au sens des coordonnées $(x^i)$
lorsque $r \to 0$). On introduit
$$\rho = \sqrt{\sum z^{i\,2}} = \frac{1}{r} \qquad
   \textrm{et} \qquad \partial_\rho = \sum_i\frac{z^i}{\rho}\partial_{z^i}=-r^2\delr\,.$$
D'après \eqref{dxdz}, si une fonction $u$ est $\grandO_{\alpha}$ lorsque $r\to 0$ (\resp{}
$\grandO_{\alpha+1,\ln}$), alors pour tout multi-indice $J$
$$\partial_z^Ju = O(\rho^{-\alpha-\abs{J}}) \qquad \textrm{(\resp{} } o(\rho^{-\alpha-\abs{J}})\textrm{)}$$
lorsque $\rho\to\infty$. Pour alléger, on notera simplement $O(\rho^{-\alpha})$ (\resp{} $o(\rho^{-\alpha})$)
une fonction ayant une telle décroissance. L'équation \eqref{asplatr} implique alors
\be \hat{g}(\partial_{z^i}, \partial_{z^j}) = \delta_{ij} + O(\rho^{-2})\,. \label{asplatrho} \ee
La métrique $\hat{g}$ est ainsi asymptotiquement plate à l'ordre $2$.

On effectue désormais les calculs dans les coordonnées $(z^i)$: par exemple $\partial_i$ désigne
$\partial_{z^i}$, $\hat{g}_{ij}=\hat{g}(\partial_{z^i},\partial_{z^j})$, \emph{etc}.
On note
$$M_R = \set{\rho\leq R} \qquad \textrm{et} \qquad S_{R} = \partial M_R=\set{\rho=R}\,.$$
La $\hat{g}$-normale unitaire sortante de $M_R$ est, d'après \eqref{daH}:
\be\hat{n} = \rho^2H^{-\frac{2}{n-2}}\partial_\rho
   =\left(1+O(\rho^{-3})\right)\partial_\rho\,.\label{normale}\ee

Les lois de changement conforme de métrique assurent que sur $\hat{M}$:
$$\wh{\Scal} = H^{-\frac{n+2}{n-2}}\Pan_1 H = 0$$
et
$$\hat{\Pan}_2 \left(H^{-\frac{n-4}{n-2}} G_2\right)
=\vphi^{\left(-\frac{n+4}{n-4}\right)}\Pan_2 \left(\vphi H^{-\frac{n-4}{n-2}}G_2\right) =
\vphi^{\left(-\frac{n+4}{n-4}\right)} \Pan_2 G_2 = 0.$$
Posons $\hat{G} = H^{-\frac{n-4}{n-2}}G_2$. La preuve du théorème \ref{theomassepos} consiste à
évaluer
\be 0= \int_{M_{R}}\hat{\Pan}_2\hat{G} \ud\hat{\vol} \label{intpg1}\ee
lorsque $R\to\infty$.

Par intégration par parties, en utilisant la formule \eqref{defP2} pour $\Pan_2$ et la nullité de $\wh{\Scal}$:
\begin{equation}
\int_{M_R} \hat{\Pan}_2\hat{G} \ud\hat{\vol} =
      -\oint_{S_R}\hat{n}\left(\hat{\Delta} \hat{G}\right) d \hat{\sigma}
      + \frac{4}{n-2}\oint_{S_{R}}\wh{\Ric}\left(\hat{\nabla} \hat{G},\hat{n}\right) d \hat{\sigma}\\
  -\frac{n-4}{(n-2)^2} \int_{M_{R}} \abs{\wh{\Ric}}_{\hat{g}}^2 \hat{G} \ud\hat{\vol}\,. \label{intpg2}
\end{equation}
On évalue les deux premiers termes du membre de droite lorsque $R \to \infty$.\\

{\it Premier terme: }
On dispose, d'après le théorème \ref{theodvpG}, du \da
$$ G_2 = \frac{1}{\clap_{n,2} r^{n-4}} + A + \grandO_{1,\ln} =
\frac{1}{r^{n-4}}(\clap_{n,2}^{-1} + Ar^{n-4} + \grandO_{n-3,\ln})\,.$$
Ce développement asymptotique implique avec celui \eqref{daH} de $H$ que
\begin{align} \hat{G} & = \clap_{n,2}^{-1}+A\rho^{4-n}-\frac{n-4}{n-2}a\rho^{-3}
      +o(\rho^{-3})+o(\rho^{4-n})\notag\\
   & = \clap_{n,2}^{-1} + A\rho^{4-n} + o(\rho^{4-n})\label{daGhat}\end{align}
car $n \leq 7$, et $a=0$ si $n>5$.
Puisqu'on a
$$\hat{\Delta} = -\sum \partial_i^2 +(\delta^{ij}-\hat{g}^{ij})\partial_i\partial_j
     +\hat{g}^{ij}\hat{\G}_{ij}^k\partial_k\,,$$
la platitude asymptotique \eqref{asplatrho} implique, de façon analogue à la partie \ref{secpkralpha}, que
$\hat{\Delta} \hat{G} = 2(n-4) A \rho^{2-n} +o(\rho^{2-n})$,
d'où avec \eqref{normale}
$$\hat{n}(\hat{\Delta}\hat{G}) = -2(n-2)(n-4) A \rho^{1-n} + o(\rho^{1-n})\,.$$
On intègre sur $S_{R}$ munie de la forme volume $d \hat{\sigma}$. En raison de la platitude asymptotique
\eqref{asplatrho}, cette dernière est asymptotiquement égale à la forme volume euclidienne de la sphère de
coordonnées $z$, $\rho^{n-1}d\sigma_{\Sph^{n-1}}$. Par conséquent,
\be \oint_{S_{R}} \hat{n}(\hat{\Delta}\hat{G}) d \hat{\sigma} = -\clap_{n,2}A
 + o(1)\,. \label{terme1}\ee

{\it Deuxième terme: }
La platitude asymptotique \eqref{asplatrho} assure que
$$|\wh{\Ric}|_{\hat{g}} = O(\rho^{-4})\,,\quad d\hat{\sigma}=(1+o(1))\rho^{n-1}d\sigma_{\Sph^{n-1}}\,,
\quad \textrm{et }|\hat{\nabla}\hat{G}|_{\hat{g}}=O(\rho^{3-n})$$
avec \eqref{daGhat} pour la dernière égalité. On a donc
$$\oint_{S_R} \wh{\Ric}(\hat{\nabla}\hat{G},\hat{n})d\hat{\sigma} \xrightarrow[R \to \infty]{} 0\,.$$

{\it Positivité de la masse:}
Ainsi, prenant $R\to\infty$ dans \eqref{intpg1}, utilisant cette dernière limite avec \eqref{intpg2} et
\eqref{terme1}, on obtient
$$\clap_{n,2} A = \frac{n-4}{(n-2)^2} \int_{\hat M}\abs{\wh{\Ric}}^2_{\hat{g}} \hat{G} \ud\hat{\vol}$$
avec $\hat{G}=H^{-\frac{n-4}{n-2}} G_2 >0$ d'après (\hyperref[h2]{h\ref*{h2}}). Cette égalité assure
donc d'une part la convergence de l'intégrale
de droite, et d'autre part la positivité (au sens large) de $A$.\\

{\it Cas d'égalité:}
Si $A=0$, $(\hat{M},\hat{g})$ est Ricci plate. L'inégalité de Gromov-Bishop assure alors que
le volume des boules croît moins vite que celui des boules de l'espace euclidien. Plus précisément,
notant $\Vol\left(B^{\hat{g}}(x,t)\right)$ le volume de la $\hat{g}$-boule de centre $x$ et de rayon
$t$, la fonction
$$t \mapsto \frac{\Vol\left(B^{\hat{g}}(x,t)\right)}{t^n\Vol(\Sph^{n-1})}$$
est décroissante. Or tant à la limite $t\to0$ que, en raison de la platitude asymptotique
\eqref{asplatrho}, à la limite $t\to\infty$, ce quotient vaut $1$. On a donc
$$\forall t>0,\quad \Vol\left(B^{\hat{g}}(x,t)\right) = t^n\Vol(\Sph^{n-1})\,.$$
Ce cas d'égalité, traité par exemple dans \cite[§ 4.20 bis]{GHL},
implique que $(M-\set{p}, \hat{g})$ est isométrique à $\R^n$. Donc $(M,[g_{0}])$ est la sphère munie de
sa classe conforme canonique.
\end{proof}

\begin{rema} Ce calcul se généralise facilement aux dimensions supérieures. On introduit la notation
$$\Pan_m^g = \Delta_g^m + \delta_g T^g_m d + \frac{n-2m}{2}Q_m^g$$
où $T^g_m$ est un opérateur différentiel de degré $2m-4$ au plus (voir la proposition \ref{propsuitepk})
et $Q_m^g$ est la Q-courbure associée à $\Pan_m^g$. Cette écriture est justifiée car $\Pan_m^g$ est autoadjoint,
\emph{cf.} \cite{GrahamZworski-scatmat}. Supposons que sur $M$, $\Pan_k^g$
et $\Pan^g_l$ soient inversibles, avec $2k+1\leq n\leq2k+3$ et $l<k$, et notons $G_k$ et $G_l$
leurs fonctions de Green respectives. On pose comme précédemment $H=\clap_{n,l}G_l^g(p,.)$,
$\hat{g}=H^{\frac{4}{n-2l}}g$ et 
$$\hat{G}=H^{-\frac{n-2k}{n-2l}}G_k$$
Alors $\hat{g}$ est asymptotiquement plate à l'ordre $2$, avec $\hat{Q}_l=0$, et
$$\hat{\Pan}_k\hat{G}=0\,.$$
Avec les mêmes notations que dans la preuve ci-dessus on a alors
\[
0 = \int_{M_R}\hat{\Pan}_k\hat{G}\ud\hat{\vol} = -\oint_{S_R}\hat{n}(\hat{\Delta}^{k-1}\hat{G})d\hat{\sigma}
      -\oint_{S_R}\hat{T}_k(d\hat{G})(\hat{n})d\hat{\sigma}
   + \frac{n-2k}{2}\int_{M_R}\hat{Q}_k\hat{G}\ud\hat{\vol}\,.
\]
Exactement comme ci-dessus, le premier terme du membre de droite tend vers $\clap_{n,k}A$ (où $A$ est
donné par le théorème \ref{theodvpG}) et le second vers $0$ (il faut jeter un \oe{}il rapide aux coefficients
de $T_k$, de façon analogue à ce qui est fait dans la preuve du théorème \ref{theoinvas}). On a donc
$$A=-\frac{n-2k}{2\clap_{n,k}}\int_{\hat{M}}\hat{Q}_k\hat{G}\ud\hat{\vol}\,.$$
On peut imaginer exploiter la nullité de $\hat{Q}_l$ dans l'intégrande du membre de droite. Mais malgré les
travaux actuels, en particulier de Juhl et ses collaborateurs \cite{Juhl-Pk,JuhlFalk-recQcuv}, les relations
entre les différentes Q-courbures sont compliquées et encore trop mal comprises pour être utilisées ici.
\label{rqmassepos}\end{rema}

\subsection{Espaces sphériques}

\subsubsection{Masse et revêtements} \label{sssecrevet}

La positivité de la fonction de Green est une conséquence du principe du maximum dans le cas du laplacien
conforme, $k=1$. Dans les autres cas, c'est une hypothèse technique. Elle permet néanmoins de montrer également
la croissance de la métrique de Habermann-Jost par revêtements conformes.
\begin{prop}\label{propmasserevet} Soient $(M,[g])$ et $(N,[h])$ deux variétés compactes de dimension $n$
munies de classes conformes, et $k$ tel que $2k+1 \leq n \leq 2k+3$. Soit
$$f: (M,[g]) \to (N,[h])$$
un revêtement conforme.
Si $\Pan_k^g: C^\infty(M) \to C^\infty(M)$ est inversible, alors
$\Pan_k^h: C^\infty(N) \to C^\infty(N)$ l'est aussi, et pour tout $p\in N$ :
\be\textrm{si }f(x)=p,\quad A^h_p = A^{f^*h}_x + \sum_{\substack{y \in f^{-1}(p)\\y \neq x}} G^{f^*h}_k(x,y)\,.
\label{relmasserevet} \ee
\end{prop}

\begin{coro} \label{corolcroismasse} Sous les mêmes hypothèses, si de plus
$$\forall x,y \in M,\quad G_k^g(x,y) > 0$$
alors pour tout $x \in M$, $A^h_{f(x)}\geq A_x^{f^*h}$. L'égalité n'est atteinte en un point que si $f$ est une
équivalence conforme. \end{coro}
\begin{coro} Sous les mêmes hypothèses, si de plus $\forall p \in M$, $A_p^g\geq 0$, alors les classes
conformes $[g]$ et $[h]$ admettent des métriques de Habermann-Jost respectives $\mathfrak{g}$ et $\mathfrak{h}$.
Celles-ci vérifient
$$f^*\mathfrak{h} \geq \mathfrak{g} $$
sur le fibré tangent $TM$. L'égalité n'est atteinte sur un vecteur non nul que si $f$ est un isomorphisme conforme,
auquel cas $f^*\mathfrak{h} = \mathfrak{g}$ partout.
\label{corolcroismethab}\end{coro}
\begin{rema} Habermann et Jost \cite{Habermann-LN,HabermannJost-methabjost} démontrent, dans le cas $k=1$, le
corollaire \ref{corolcroismethab} à l'aide d'une formule analogue à \eqref{relmasserevet}, qui n'est valable
que pour un revêtement galoisien. Le cas général nécessite quelques détails supplémentaires.
\end{rema}

\begin{proof}[Démonstration des corollaires]
Ce sont des conséquences élémentaires de la proposition. Par covariance conforme, la positivité
de $G^{f^*h}_k$ équivaut à celle de $G_k^g$.
La formule \eqref{relmasserevet} assure donc l'inégalité stricte du corollaire \ref{corolcroismasse}, à moins
que les fibres de $f$ ne se réduisent à un point, auquel cas on a égalité pour tout $x$.

En particulier si $A^g_x \geq 0$ pour tout $x \in M$, alors par covariance conforme $A^{f^*h}_x > 0$, donc
$A^h_p>0$ pour tout $p \in M$ (le revêtement $f$ entre variétés compactes connexes étant nécessairement surjectif). 
Appliquant la définition \ref{defmetHab}:
$$\forall x \in M,\quad \mathfrak{g}_x = \left(A_p^{f^*h}\right)^{\frac{2}{n-2k}} (f^*h)_x \quad
\textrm{et} \quad (f^*\mathfrak{h})_x = \left(A_{f(x)}^h\right)^{\frac{2}{n-2k}}(f^*h)_x\,.$$
Le corollaire \ref{corolcroismethab} est alors une application directe du précédent.
\end{proof}

\begin{proof}[Démonstration de la proposition \ref{propmasserevet}] On choisit
$g=f^*h$.
L'idée est de prouver que, pour $p,q \in N$, si on fixe un antécédent $x$ de $p$ par $f$,
$$H(p,q)=\sum_{y\in f^{-1}(q)}G_k^g(x,y)$$
définit (le noyau intégral d') un inverse de $\Pan_k^h$. C'est élémentaire si $f$ est un revêtement
galoisien. Dans le cas contraire, il n'est pas évident que cette définition soit indépendante
de $x$. On s'en convainc en écrivant $M$ et $N$ comme quotients d'une même variété compacte.

Détaillons. Soient $\tilde\G_1$ et $\tilde\G_2$ les groupes fondamentaux respectifs de $M$ et $N$ (pour
un choix convenable de points-bases qui importe peu ici). \emph{Via} $f$, le groupe $\tilde\Gamma_1$ est
identifié à un sous-groupe d'indice fini de $\tilde\Gamma_2$. Le noyau $\Gamma_0$ de l'action à droite
de $\tilde\Gamma_2$ sur $\lcoset{\tilde\Gamma_1}{\tilde\Gamma_2}$ (\cad{} l'ensemble des éléments qui agissent
comme la permutation identité) est un sous-groupe distingué d'indice fini dans $\tilde\G_2$, inclus dans
$\tilde\G_1$ qui est le stabilisateur d'une classe à gauche. D'après la théorie galoisienne des revêtements
(voir \cite{polyLannes} par exemple), les deux flèches
$$\xymatrix{ & E= \lcoset{\Gamma_0}{\tilde{E}} \ar^{\pi_2}[dr] \ar_{\pi_1}[dl]& \\
  M=\lcoset{\tilde\Gamma_1}{\tilde{E}} & & N = \lcoset{\tilde\Gamma_2}{\tilde{E}}
}$$
où $\tilde E$ est le revêtement universel de $M$ et $N$, sont des revêtements galoisiens à nombres de feuillets
finis. Soient $\G_1$ et
$\G_2$ leurs groupes d'automorphismes respectifs. On munit $E$ de la métrique $e=\pi_2^*h = \pi_1^*g$.

L'opérateur $\Pan_k^e$ pourrait ne pas être inversible. Néanmoins, autoadjoint (voir
\cite{GrahamZworski-scatmat}) et elliptique sur la variété riemannienne compacte $(E,e)$, il admet un unique
pseudo-inverse $G_k^e$, tel que
$$\Pan_k^e G_k^e = G_k^e \Pan_k^e = \Id - \Pi^e$$
où $\Pi^e$ est le projecteur $L^2$-orthogonal (au sens de $e$) sur $\ker \Pan_k^e$. On identifiera opérateur et
noyau intégral dans la suite. Ainsi $G_k^e$ est une fonction $C^\infty$ sur $E \times E$ privée de
sa diagonale.

Pour $a ,b \in E$ dans des $\G_1$-orbites différentes, on pose
$$H_E(a,b)=\sum_{\gamma \in \G_1} G_k^e(a,\gamma.b)\,.$$
C'est clairement un opérateur $C^\infty(E) \to C^\infty(E)$. Le groupe $\G_1$ agissant par isométries
de $e$, $H_E$ est $\G_1\times\G_1$-invariante et passe donc \og{}au quotient\fg{} en une fonction
$H_M$ sur $M\times M$ privé de sa diagonale, de classe $C^\infty$.

Soit $U$ un ouvert de $M$ tel que $\pi_1^{-1}(U)$ est la réunion des ouverts disjoints
$V_1,\dotsc, V_d \subset E$, et que pour chacun d'eux $\pi_1|_{V_j}: V_j \to U$ est une isométrie.
On note
$$\pi_{1,j}^{-1} = \left(\pi_1|_{V_j}\right)^{-1}: U \to V_j\,.$$
Soit $u$ une fonction $C^\infty$ sur $M$ à support dans $U$. Le support de $u\circ\pi_1$ est
inclus dans $\bigcup_jV_j$; on a donc pour tout $a \in E$
\begin{align*}
u(\pi_1(a))-\Pi^e(u\circ \pi_1)(a) & = \int_M G_k^{e}(a,b) \Pan_k^{e}\big(u\circ \pi_1\big)(b) \vol_{e}(\ud b)\\
   & = \sum_j \int_{V_j} G_k^{e}(a,b) \Pan_k^{e}\big(u \circ \pi_1\big)(b)\vol_{e} (\ud b) \\
   & = \sum_j \int_U G_k^{e}\big(a,\pi_{1,j}^{-1}(y)\big) \Pan_k^gu(y) \vol_g(\ud y)\\
\intertext{avec le changement de variables $b=\pi_{1,j}^{-1}(y)$ sur chaque $V_j$. D'où}
u(\pi_1(a))-\Pi^e(u\circ \pi_1)(a) & = \int_U H_E\left(a,\pi_{1,1}^{-1}(y)\right) \Pan_k^g(u)(y) \vol_g(\ud y)\\
   & = \int_M H_M(\pi_1(a),y) \Pan_k^g(u)(y) \vol_g(\ud y)\,.
\end{align*}
Par ailleurs,
$\Pi^e(u\circ\pi_1)$ est une fonction de $\ker \Pan_k^e$, $\G_1$-invariante de manière
évidente. Elle passe donc \og{}au quotient\fg{} en une fonction sur $M$ qui annule $\Pan_k^g$. Comme
on fait l'hypothèse que cet opérateur est inversible, on a
$$\Pi^e(u\circ \pi_1)=0\,.$$
D'où, pour $x=\pi_1(a)$:
$$u(x) = \int_M H_M(x,y)\Pan_k^gu(y) \vol_g(\ud y)\,.$$
Par linéarité, à l'aide d'une partition de l'unité, ce résultat se généralise à toute fonction
$C^\infty$ sur $M$. La symétrie de $H_M$ en $x$ et $y$ et le fait que $\Pan_k^g$ est autoadjoint
impliquent que $H_M$ est aussi un inverse à droite de $\Pan_k^g$. D'où
\be H_M(x,y) = G_k^g(x,y) = \sum_{\gamma \in \G_1}G_k^e(a,\gamma.b) \label{Grevet} \ee
pour n'importe quels $a,b$ dans les fibres respectives de $x$ et de $y$.

Pour toute fonction $u \in \ker \Pan_k^h$ sur $N$, on a $u \circ f \in \ker \Pan_k^g = \set{0}$ sur $M$.
L'opérateur $\Pan_k^h$, fredholmien d'indice $0$, est donc inversible. La formule \eqref{Grevet} vaut alors aussi
pour le revêtement $\pi_2: E \to N$; on a donc, pour
tous $p,q \in N$, $a,b\in E$ dans leurs fibres respectives par $\pi_2$, et tout $x \in M$ au-dessus
de $p$:
\begin{align*}
G_k^h(p,q) & = \sum_{\gamma_2 \in \G_2} G_k^e(a,\gamma_2.b) \\
   & = \sum_{[s] \in \G_1 \backslash \G_2} \sum_{\gamma_1\in \G_1} G_k^e(a,\gamma_1.s.b)\\
   & = \sum_{[s] \in \G_1 \backslash \G_2} G_k^g(x,\pi_1(s.b))\\
   & = \sum_{y \in f^{-1}(q)} G_k^g(x,y)\,.
\end{align*}
On choisit $h$ vérifiant \eqref{ricnormalconf} en $p$; alors $f^*h$ est normale conforme à l'ordre $4$ en
$x$. On fait tendre $q$ vers $p$. Le comportement asymptotique énoncé par le théorème \ref{theodvpG}
fournit alors \eqref{relmasserevet}. Le cas où
$h$ n'est pas normale conforme s'en déduit par covariance conforme.
\end{proof}

\subsubsection{Quotients de la sphère} \label{sssecqsphere}

La proposition \ref{propmasserevet} permet de calculer explicitement la métrique de Haber\-mann-Jost
des quotients d'une sphère $\Sph^n$ par un sous-groupe fini de $SO(n+1)$ agissant librement.

Les fonctions de Green des opérateurs $\Pan_k$ sur la sphère ronde sont connus: voir Branson
\cite{Branson-funcdet}. Comme elles n'apparaissent dans cette référence qu'à constante multiplicative près,
on les recalcule ici de manière simple. Soient
$$\xi = (\xi^i)_{0\leq i\leq n}: \Sph^n \to \R^{n+1}$$
l'inclusion canonique, $g_\circ = \xi^*\eucl$ la métrique ronde, et $N=(1,0\dots0)$ le pôle nord. On note
$$\phi: \Sph^n-\set{N} \to  \R^n$$
la projection stéréographique. C'est un isomorphisme conforme:
$$\phi_*g_\circ = \left(\frac{2}{1+\norm{x}^2}\right)^2\sum_{i=1}^ndx^{i\,2}\,.$$
Sur $\R^n$ muni de la métrique plate $\eucl$, il est connu que
$$G_k^{\eucl}(x,y)=\frac{1}{\clap_{n,k}} \norm{x-y}^{2k-n}$$
est, disons, un inverse à gauche de $\Pan_k^\eucl=\Delta_\eucl^k$ sur l'espace des fonctions
lisses à support compact. Par covariance conforme, on définit un inverse à gauche de $\Pan_k^{g_\circ}$
sur l'espace des fonctions de classe $C^\infty$ sur $\Sph^n$ nulles au voisinage de $N$ par
\begin{align}
G_k^{g_\circ}(p,q)&=\frac{1}{\clap_{n,k}}\left(\frac{2}{(1+\norm{\phi(p)}^2)}
   \frac{2}{(1+\norm{\phi(q)}^2)}\norm{\phi(p)-\phi(q)}^2\right)^{\frac{2k-n}{2}}\notag\\
   & = \frac{1}{\clap_{n,k}} \norm{\xi(p)-\xi(q)}^{2k-n}\,. \label{Gsph}
\end{align}
(Le passage de la première à la seconde ligne est classique. C'est une conséquence de la proposition
\ref{proprconfplat}, puisque la projection stéréographique est la restriction à $\Sph^n$ de l'inversion
par rapport à la sphère de centre $N$ et de rayon $\sqrt{2}$.) Par invariance sous $SO(n+1)$ et symétrie
en $p$ et $q$, la formule \eqref{Gsph} définit un inverse de $\Pan_k^{g_\circ}$.

L'existence de cette fonction de Green assure, d'après le théorème \ref{theodvpG}, que la masse des
opérateurs $\Pan_k^{g_\circ}$ est bien définie. Le groupe $SO(n+1)$ agissant transitivement par isométries,
la fonction
$$p \mapsto A_p^{g_\circ}$$
est constante. Si elle était non nulle, quitte à prendre sa valeur absolue, elle permettrait de définir une
métrique de Habermann-Jost sur $\Sph^n$. Cette métrique serait invariante par le groupe conforme
de la sphère, ce qui n'est pas possible car celui-ci n'est pas compact. On en déduit donc:
\begin{prop} La masse des opérateurs GJMS sur la sphère est nulle.
\end{prop}
Le calcul de la fonction de Green de $\Pan_k^{g_\circ}$ et la proposition \ref{propmasserevet} donnent
alors la métrique de Habermann-Jost des espaces sphériques.
\begin{prop} Soit $\G$ un sous-groupe fini de $SO(n+1)$ agissant librement par isométries sur la sphère.
Soit $g_0$ la métrique ronde induite sur $\lcoset{\G}{\Sph^n}$. La métrique de Habermann-Jost
de $[g_0]$ est
$$\mathfrak{g}_0(p)=\left(A_p^0\right)^{\frac{2}{n-2k}} g_0(p)$$
où, pour n'importe quel élément $\xi$ de la fibre au-dessus de $p$:
$$A_p^0=\frac{1}{\clap_{n,k}}\sum_{R \in \G-\set{\Id}} \norm{(R-\Id)\xi}^{2k-n}\,.$$
\label{propquosphere} \end{prop}
Cette métrique n'est pas multiple de la métrique ronde en général: Habermann \cite{Habermann-LN} le
montre dans le cas $k=1$ pour l'espace lenticulaire $L(5,2)$; pour l'espace $L(7,2)$ il calcule que la
métrique de Habermann-Jost a une courbure scalaire de signe variable.

\section{Masses et invariants asymptotiques} \label{partieinvas}

On démontre ici le lien entre la masse de l'opérateur GJMS $\Pan_k^g$ et la masse $m_k$,
introduite dans la définition \ref{defmasseas}, de la métrique obtenue depuis $g$ par
\og{}projection stéréographique\fg{} à l'aide de la fonction de Green $G_k^g$.

\begin{theo} \label{theoinvas}
Dans le cas $2k+1 \leq n \leq 2k+3$, supposons que $\Pan_k^g$ admet une fonction de Green $G_k$.
Pour $p\in M$, soit $G_p$ une fonction lisse positive sur $M-\set{p}$ égalant $G_k(p,.)$ au voisinage de $p$.
Soit $\hat{g}=G_p^{\frac{4}{n-2k}}g$ sur $M-\set{p}$.

Alors
\begin{align*}\Delta_{\hat{g}}^{k-1}\Scal^{\hat{g}} & \in L^1(M-\set{p},\ud\vol_{\hat{g}})\\
\intertext{et}
A_p^g&=\frac{n-2k}{4(n-1)}m_k(\hat{g})\,.\end{align*}
\end{theo}
\begin{rema}
\begin{enumerate}[(i)]
\item Il s'avère au cours de la preuve que la métrique $\hat{g}$ est asymptotiquement plate à l'ordre
$1$ si $n=2k+1$ et à l'ordre $2$ si $n\geq 2k+2$. Dans le cas $k=1$, le théorème reste donc valable,
en prenant pour $m_1$ la masse ADM rappelée en introduction. Ce cas était déjà connu, voir \cite{LeeParker-Yamabeprob,%
Habermann-LN} qui ont des normalisations différentes pour la masse ADM.
\item Si $g_2=\vphi^{\frac{4}{n-2k}}g_1$, les fonctions de Green des opérateurs $\Pan_k$ respectifs sont reliées
par
\begin{align*}
G_k^{g_2}(p,x)&=\frac{1}{\vphi(p)\vphi(x)}G_k^{g_1}(p,x)\,,\\
\intertext{donc les métriques \og{}explosées\fg{} sont multiples l'une de l'autre par un facteur constant:}
\hat{g}_2&=\vphi(p)^{\frac{-4}{n-2k}}\hat{g}_1\,.\\
\intertext{On a donc}
\Delta_{\hat{g}_2}^{k-1}\Scal^{\hat{g}_2} &= \vphi(p)^{\frac{4k}{n-2k}}\Delta_{\hat{g}_1}^{k-1}\Scal^{\hat{g}_1}\\
\intertext{et}
\ud\vol_{\hat{g}_2} &= \vphi(p)^{\frac{-2n}{n-2k}}\ud\vol_{\hat{g}_1}\,,\\
\intertext{d'où}
m_k(\hat{g}_2)&=\vphi(p)^{-2}m_{k}(\hat{g}_1)\,.
\end{align*}
Le théorème \ref{theoinvas} fournit ainsi une autre preuve de la covariance
conforme de la masse (théorème \ref{theocovconf}) dans le cas $k\geq 2$. (Comme remarqué dans \cite{Habermann-LN},
il en va de même pour $k=1$, car  la masse ADM est aussi homothétiquement équivariante, de poids $n-2$).
\end{enumerate}
\end{rema}

\begin{proof}[Démonstration du théorème \ref{theoinvas}]

D'après cette dernière remarque, quitte à modifier $g$ d'un facteur con\-forme valant $1$ en $p$,
ce qui ne modifie pas $\hat g$, on peut
supposer que $g$ est normale conforme à l'ordre $4$ en $p$, et appliquer le théorème \ref{theodvpG}:
$$G_k(p,x)=\clap_{n,k}^{-1}r^{2k-n}(1 + \clap_{n,k}A r^{n-2k} + \grandO_{n-2k+1,\ln})\,.$$
On notera $\clap=
\clap_{n,k}$ et, pour alléger, on fera les calculs avec
$$\bar{g}=\left(\clap G_p\right)^{\frac{4}{n-2k}}g =\clap^{\frac{4}{n-2k}}\hat{g}\,.$$

On reprend l'idée et les notations de la démonstration du théorème \ref{theomassepos}.
Soient $(x^i)_i$ des cordonnées $g$-géodésiques en $p$. Au voisinage de $p$
\begin{align*}
g(\partial_{x^i},\partial_{x^j}) & = \delta_{ij} -\frac{1}{3} \mathcal{R}_{iklj}x^kx^l+ O(r^3)\,,\\
\intertext{où}
\mathcal{R}_{iklj} &= g(\Rm^g_{(\partial_{x^i},\partial_{x^k})}\partial_{x^l},\partial_{x^j})(p)\,.
\end{align*}
On utilise les coordonnées inversées $z^i=x^i/r^2$ et $\rho=1/r$. D'après les liens \eqref{dxdz} et
\eqref{gz} entre coordonnées $z$ et coordonnées $x$, on obtient pour $\rho\to\infty$ le \da{}
\be \begin{split} \bar{g}(\partial_{z^i},\partial_{z^j}) =\delta_{ij}
       -\frac{1}{3} \rho^{-4}\mathcal{R}_{iklj}&z^kz^l + O(\rho^{-3}) \\
   &+\frac{4\clap}{n-2k}A\rho^{2k-n}\delta_{ij}+o(\rho^{2k-n})\,.
\end{split}\label{dagc} \ee
   
On calcule désormais en coordonnées dans la carte $z^i$, $\partial_i$ désignant
$\partial_{z^i}$, $\bar{g}_{ij}=\bar{g}(\partial_{z^i},\partial_{z^j})$, \emph{etc}. Le développement
asymptotique \eqref{dagc} assure que $\bar{g}$ est asymptotiquement plate à l'ordre $\tau$,
avec $\tau=1$ si $n=2k+1$, et $\tau=2$ sinon.\\

{\it Intégrabilité:}
On note conformément à l'habitude
$$\bar{Q}_k = \frac{2}{n-2k}\bar{\Pan}_k(1)\,.$$
Cette Q-courbure s'écrit comme combinaison linéaire de contractions
totales (\emph{via} $\bar g$) de tenseurs de la forme
\be\bar{\nabla}^{d_1}\bar{\Rm}\otimes\cdots \otimes\bar{\nabla}^{d_a}\bar{\Rm}\,.\label{termeqk}\ee
Le comportement de $\Pan_k$ sous les changements homothétiques de métrique impose
que pour chacun de ces tenseurs apparaissant dans $\bar{Q}_k$, $\sum_b(d_b+2)=2k$. 
Comme $\bar\nabla^i\bar\Rm\bas{kij}\haut{k} = \frac{1}{2}\partial_j\overline\Scal$
d'après l'identité de Bianchi, et que
$$\bar\nabla^i\bar\nabla^j\bar\Rm\bas{ijk}\haut{l} =\frac{1}{2}[\bar\nabla^i,\bar\nabla^j]\bar\Rm\bas{ijk}\haut{l}$$
est quadratique en la courbure d'après la formule de Ricci, la seule des contractions totales de tenseurs de
la forme \eqref{termeqk} qui soit linéaire
en la courbure est
$\bar\Delta^{k-1}\overline{\Scal}$. On a donc
$$\bar{Q}_k=c_k\bar\Delta^{k-1}\overline{\Scal} + \mathscr{Q}_k(\bar\Rm)$$
où $\mathscr{Q}_k(\bar\Rm)$ est un polynôme sans terme linéaire en la courbure. En comparant les linéarisations
sur les variations conformes de la métrique plate il vient
$$c_k=\frac{1}{2n-2}\neq 0\,,$$
voir Branson
\cite[Corollaire 1.5]{Branson-funcdet} pour une formulation plus évoluée.

D'après la loi de variation conforme de $\Pan_k$ et le choix de facteur conforme, on a, du moins au
voisinage de l'infini, $\bar Q_k=0$. Donc
$$\bar{\Delta}^{k-1}\overline{\Scal} = -\frac{1}{c_k}\mathscr{Q}_k(\bar{\Rm})\,.$$
Les termes \eqref{termeqk} qui apparaissent dans $\mathscr{Q}_k(\bar\Rm)$ vérifient tous
$a\geq2$. En raison de la platitude asymptotique on a donc
$$\abs{\bar{\nabla}^{d_1}\bar{\Rm}\otimes\cdots \otimes\bar{\nabla}^{d_a}\bar{\Rm}}_{\bar{g}} =
O(\rho^{-\tau-d_1-2 \cdots - \tau-d_a-2}) = O(\rho^{-a\tau-2k})=O(\rho^{-2\tau-2k})\,.$$
Or $-2\tau-2k<-n$, et $\ud\bar{\vol}\sim dz^1dz^2\dots dz^n$ lorsque $\rho \to \infty$;
on déduit donc
$$\bar{\Delta}^{k-1}\overline{\Scal} \in L^1(M-\set{p},\ud\bar{\vol})\,.$$

{\it Calcul de l'intégrale:} D'après le théorème de la divergence, si $k \geq 2$,
$$\int_{\set{\rho \leq R}} \bar{\Delta}^{k-1}\overline{\Scal}\,\ud\bar{\vol} =
  -\oint_{\set{\rho=R}}\bar{n}(\bar{\Delta}^{k-2}\overline{\Scal})d\bar{\sigma}$$
où $\bar{n}$ est la normale unitaire de $S_R=\set{\rho=R}$ pointant vers l'infini, et
$d\bar{\sigma}$ la mesure de volume sur $S_R$ induite par $\ud\bar{\vol}$.
La platitude asymptotique de $\bar{g}$ implique que
\begin{align*}
\bar{n}&= \left(1+O(\rho^{-\tau})\right)\partial_\rho
   = \left(1+O(\rho^{-\tau})\right)\frac{z^i}{\rho}\partial_i\,,\\
\bar{\Delta} & = -\sum_i\partial_i^2 + O(\rho^{-\tau})\partial_i\partial_j + O(\rho^{-\tau-1})\partial_i\,,\\
\overline{\Scal} & = \partial_j(\partial_i\bar{g}_{ij}-\partial_j\bar{g}_{ii}) + O(\rho^{-2\tau-2})\\
\intertext{et}
d\bar{\sigma}&= (1 + O(\rho^{-\tau}))d\sigma_R\,,
\end{align*}
où $d\sigma_R$ est la mesure de volume sur $S_R$ induite par
la métrique plate à l'infini $\sum_i dz^{i\,2}$. Ainsi
$$\bar{n}(\bar{\Delta}^{k-2}\overline{\Scal})d\bar{\sigma} = \left(\frac{z^l}{\rho}\partial_l\Delta_0^{k-2}
\partial_j (\partial_i\bar{g}_{ij}-\partial_j\bar{g}_{ii})+O(\rho^{-2\tau-2k+1})\right)d\sigma_R\,.$$
où, comme dans la partie \ref{secpkralpha}, $\Delta_0=-\sum_i\partial_i^2$.
Puisque $\Vol(S_R,d\sigma_R) = O(R^{n-1})$ avec $n-1<2\tau+2k-1$, il s'en déduit
\be \int_{M-\set{p}}\bar{\Delta}^{k-1}\overline{\Scal}\,\ud\bar{\vol} = -\lim_{R\to \infty}\oint_{S_R}
   \frac{z^l}{\rho} \partial_l \Delta_0^{k-2}\partial_j(\partial_i\bar{g}_{ij}-\partial_j\bar{g}_{ii}) d\sigma_R\,.
   \label{formulemk}\ee

Utilisant le \da{} \eqref{dagc}, en remarquant que
$$\partial_i\left(\rho^{-4}\mathcal{R}_{ilmj}z^lz^m\right) = 0 = \partial_j\left(\rho^{-4}
  \mathcal{R}_{ilmi}z^lz^m\right)$$
en vertu des propriétés algébriques d'un tenseur de Riemann lorsque $\mathcal{R}_{ilmi}=0$ (ce qui est le cas
lorsque $g$ est normale conforme à l'ordre $4$ en $p$), on calcule
\begin{multline*}
\frac{z^l}{\rho}\partial_l \Delta_0^{k-2} \partial_j (\partial_i\bar{g}_{ij}-\partial_j\bar{g}_{ii})=\\
   -4(n-1)\clap A 2^{k-1}(k-1)!(n-2)(n-4)\dotsm(n-2k+2) \rho^{1-n}
   + O(\rho^{1-n-\tau})\,.
\end{multline*}

Reportant dans \eqref{formulemk}, avec $\Vol(S_R,d\sigma_R) = R^{n-1}\Vol(\Sph^{n-1})$:
\begin{align*}
\int_{M-\set{p}}\bar{\Delta}^{k-1}\overline{\Scal}\,\ud\bar{\vol} & = 
      4(n-1)\clap A\Vol(\Sph^{n-1})2^{k-1}(k-1)!
   (n-2)(n-4)\cdots(n-2k+2)\\
   & = 4(n-1) \clap \frac{\clap}{n-2k}A \,.
\end{align*}
Notons que le calcul précédent dans le cas $k=1$ fournit le même résultat, en remplaçant le membre de gauche
par la masse ADM de $\bar{g}$.

L'énoncé du théorème concerne $\hat{g}=\clap^{\frac{-4}{n-2k}}\bar{g}$. Par homothétie
$$\int_{M-\set{p}}\hat{\Delta}^{k-1}\wh{\Scal}\,\ud\hat{\vol}
   =\clap^{-2}\int_{M-\set{p}}\bar{\Delta}^{k-1}\overline{\Scal}\,\ud\bar{\vol}
   =\frac{4(n-1)}{n-2k}A\,,$$
ce qui est énoncé par la proposition.
\end{proof}

\begin{rema} \begin{enumerate}[(i)]
\item Le raisonnement concernant l'intégrabilité montre plus généralement que si $\hat{g}$ est
asymptotiquement plate à l'ordre $\tau>\frac{n-2k}{2}$, alors on a l'équivalence
$$\hat\Delta^{k-1}\wh\Scal \in L^1(\ud\hat\vol) \quad \Leftrightarrow \quad \hat{Q}_k \in L^1(\ud\hat\vol)\,.$$
La platitude asymptotique n'est néanmoins pas nécessaire pour définir $m_k$, en particulier parce que
les problèmes de changements de coordonnées asymptotiquement plates à l'infini ne se posent pas.
\item Comme expliqué à la remarque \ref{rqmassepos}, les projections stéréographiques utilisant les fonctions
de Green $G_l^g$ pour $n>2l+3$ sont également intéressantes.
D'après le théorème \ref{theodvpG}, il y a un \da{} \eqref{dagc} valable, en
retirant le terme $\frac{4\clap}{n-2k}A\rho^{2k-n}\delta_{ij}+o(\rho^{2k-n})$. La métrique $\hat{g}$
est alors asymptotiquement plate d'ordre $\tau=2$, et de même que dans la preuve, la nullité de
$\hat Q_l$ implique que $\hat{\Delta}^{l-1}\overline{\Scal} = O(\rho^{-4-2l})$ (toujours avec les mêmes
dominations asymptotiques des dérivées). Donc si $k\geq l$,
$$\hat{\Delta}^{k-1}\wh{\Scal}=\hat{\Delta}^{k-l}\hat{\Delta}^{l-1}\wh{\Scal} = O(\rho^{-4-2k})\,,$$
ce qui est $L^1$ si $n$ égale $2k+1$, $2k+2$ ou $2k+3$. Le même calcul que la dans la preuve ci-dessus, en
substituant $0$ à $A$, montre que
$$m_k(\hat g)=0\,.$$
\end{enumerate}
\end{rema}

\begin{merci}
Ce travail a été effectué lors de ma thèse à l'Institut de Mathématiques et Modélisation de Montpellier,
Université Montpellier 2, sous la direction de Marc Herzlich. Je le remercie chaleureusement pour les
discussions que j'ai eues avec lui, ses remarques et ses encouragements.
\end{merci}

\bibliographystyle{amsplain}
\bibliography{bibmassepk}

\end{document}